\documentclass[journal]{IEEEtran}%

\usepackage{amsmath,amssymb,amsfonts}%
\usepackage{theorem}%
\usepackage{graphicx}%
\usepackage{color}%
\usepackage{dsfont}%

\newtheorem{definition}{Definition}[section]%
\newtheorem{theorem}[definition]{Theorem}%
\newtheorem{proposition}[definition]{Proposition}%
\newtheorem{lemma}[definition]{Lemma}%
\newtheorem{assumption}[definition]{Assumption}%
\newtheorem{corollary}[definition]{Corollary}%
{\theorembodyfont{\rmfamily} \newtheorem{remark}[definition]{Remark}}%
{\theorembodyfont{\rmfamily} \newtheorem{example}[definition]{Example}}%

\newcommand{\im}{\mathrm{im}}%
\newcommand{\id}{\mathrm{id}}%
\newcommand{\Gl}{\mathrm{Gl}}%
\newcommand{\tm}{\times}%
\newcommand{\unit}{\mathds{1}}%
\newcommand{\esssup}{\operatorname*{ess\;sup}}%
\newcommand{\essinf}{\operatorname*{ess\;inf}}%
\newcommand{\inv}{\mathrm{inv}}%
\newcommand{\rmd}{\mathrm{d}}%
\newcommand{\rmD}{\mathrm{D}}%
\newcommand{\N}{\mathbb{N}}%
\newcommand{\Z}{\mathbb{Z}}%
\newcommand{\R}{\mathbb{R}}%
\renewcommand{\P}{\mathbb{P}}%
\newcommand{\BC}{\mathcal{B}}%
\newcommand{\FC}{\mathcal{F}}%
\newcommand{\LC}{\mathcal{L}}%
\newcommand{\MC}{\mathcal{M}}%
\newcommand{\SC}{\mathcal{S}}%
\newcommand{\UC}{\mathcal{U}}%
\newcommand{\VC}{\mathcal{V}}%
\newcommand{\WC}{\mathcal{W}}%

\allowdisplaybreaks

\begin{document}

\title{Invariance Properties of Controlled Stochastic Nonlinear Systems under Information Constraints\footnote{This research was supported in part by the Natural Sciences and Engineering Research Council (NSERC) of Canada. A brief presentation of some of the results in this paper will appear at the 2019 Information Theory Workshop, Visby, Sweden.}}

\author{Christoph Kawan and Serdar Y\"{u}ksel\thanks{C.~Kawan is with the Institute for Informatics at the Ludwig-Maximilians-Universit\"{a}t Munich, 80538 Munich, Germany (email: christoph.kawan@lmu.de). S.~Y\"uksel is with the Department of Mathematics and Statistics, Queen's University, Kingston, Ontario, Canada, K7L 3N6 (e-mail: yuksel@mast.queensu.ca). This research was supported in part by the Natural Sciences and Engineering Research Council (NSERC) of Canada. Some results of this paper were presented without proofs at the 2019 IEEE Information Theory Workshop.}}%

\maketitle

\begin{abstract}
Given a stochastic nonlinear system controlled over a possibly noisy communication channel, the paper studies the largest class of channels for which there exist coding and control policies so that the closed-loop system is stochastically stable. The stability criterion considered is asymptotic mean stationarity (AMS). We develop a general method based on ergodic theory and probability to derive fundamental bounds on information transmission requirements leading to stabilization. Through this method we develop a new notion of entropy which is tailored to derive lower bounds for asymptotic mean stationarity for both noise-free and noisy channels. The bounds obtained through probabilistic and ergodic-theoretic analysis are more refined in comparison with the bounds obtained earlier via information-theoretic methods. Moreover, our approach is more versatile in view of the models considered and allows for finer lower bounds when the AMS measure is known to admit further properties such as moment bounds.
\end{abstract}


\section{Introduction}

Consider the following problem: Given a stochastic nonlinear system controlled over a communication channel, what is the largest class of such channels so that there exist coding and control policies leading to (some form of) stochastic stability? Various versions of this problem have been studied extensively for (possibly stochastic) linear systems and {\it deterministic} nonlinear systems. 

For deterministic nonlinear systems, invariance entropy \cite{CKa} measures the smallest average data rate of a noiseless channel above which a compact subset $Q$ of the state space can be made invariant by a controller receiving its state information through this channel. The essence of the idea behind this concept is as follows: If the controller has $n$ bits of information available, it can distinguish at most $2^n$ different states, hence generate at most $2^n$ different control inputs. Consequently, the number of control inputs needed to achieve the control objective (on a finite time interval) is a measure for the necessary information. The definition of invariance entropy thus reads%
\begin{equation*}
  h_{\inv}(Q) := \lim_{\tau \rightarrow \infty}\frac{1}{\tau}\log r_{\inv}(\tau,Q),%
\end{equation*}
where $r_{\inv}(\tau,Q)$ is the minimal number of control inputs needed to achieve invariance of $Q$ on the time interval $[0,\tau]$ for arbitrary initial states in $Q$. It is relatively immediate to observe that the growth rate of $r_{\inv}(\tau,Q)$ is directly related to the rate of volume expansion for subsets of $Q$ under the evolution of the system. Indeed, the faster volume is expanded, the more coding regions, and hence different control inputs, are necessary to keep the whole volume inside $Q$. Since for every reasonable stabilization objective it is necessary to keep certain volumes bounded (or even shrink them to zero), the same ideas as used in the definition of invariance entropy should work universally for stabilization over discrete channels. This intuition was rigorously verified in a number of publications, including \cite{Col,CKa,CKN,DKa,Kaw,KDe}.%

In this paper, we demonstrate that such an approach is also applicable, by means of the machinery we develop, to stochastic systems, stochastic channels, and to stochastic stability. Our criterion for stochastic stability is asymptotic mean stationarity (AMS), introduced by Gray \& Kieffer \cite{GrayKieffer} and used in networked control in a number of publications \cite{YukselBasarBook,YukselAMSITArxiv,yuksel2016stability}. This concept considerably weakens the notion of stationarity and is closely related to other criteria used in the literature, such as \emph{stability in probability} \cite{MatveevSavkin}, \emph{(unique) ergodicity} \cite{YukselBasarBook}, as well as another commonly used stability criterion: \emph{finite $m$-th moment stability} for various $m \in \N$ \cite{Minero,NairEvans,SahaiParts}. The AMS property is weaker than unique ergodicity, and the finite-moment stability typically implies the AMS property provided additional regularity properties are imposed. Nonetheless, the AMS property is a very versatile notion; if one assumes that the support of the asymptotic mean measure is compact, the AMS property can be related to set stability; if one assumes that this measure has a finite $m$-th moment for its coordinate state process, the AMS property would lead to the finite-moment stability property, and finally the ergodicity can also be imposed for certain applications through mixing properties, e.g., through the construction of a positive Harris recurrent Markov chain \cite{YukMeynTAC2010}. Barron \cite{Barron} and Gray \& Kieffer \cite{GrayKieffer} note various other operational utilities of the AMS property.%

As an auxiliary quantity to derive lower bounds on the necessary channel capacity for generating an AMS state process, we introduce a new concept of stabilization entropy inspired by both invariance entropy and measure-theoretic entropy of dynamical systems, in particular by a characterization of the latter due to Katok \cite{Kat} and a generalization thereof developed in Ren et al.~\cite{Re2}. Roughly speaking, stabilization entropy looks at the exponential growth rate of the number of length-$n$ control sequences necessary to keep the state inside some set for a certain fraction of the number $n$ of times with a certain positive probability. The corresponding set, the frequency of times and the probability are parameters that can be adjusted, and the relation to channel capacity can only be established for certain choices of these parameters.%

Stochastic stabilization of nonlinear systems driven by noise (especially unbounded noise) over communication channels has been studied in few publications, notably in \cite{yuksel2016stability}. With our method we are able to refine the bounds presented in \cite{yuksel2016stability}. The approach developed in our paper, unlike the differential-entropic methods in \cite{yuksel2016stability} and other publications, allows for%
\begin{itemize}
\item[(i)] refined stochastic stability results applicable to a more general class of system models (Theorems \ref{thm_volexp} and \ref{thm_ibs}). and more refined stability criteria such as the AMS property in combination with moment conditions (see Corollary \ref{cor_momentcond}),%
\item[(ii)] a more concise and direct derivation, building on volume growth arguments, applicable to a plethora of criteria,%
\item[(iii)] more refined bounds for a large class of systems through trading-off growth rates with the measures of sets under the coordinate projection of a stationary measure (see Theorem \ref{thm_volexp}),%
\item[(iv)] the unification of the theory developed for deterministic systems controlled over noise-free communication channels with their stochastic counterparts, involving both stochastic nonlinear dynamical systems and noisy communication channels (see Theorem \ref{thm_noisychannel}).%
\end{itemize}

In the paper at hand, explicit lower bounds on the capacity in terms of characteristics of the system are derived for nonlinear volume-expanding systems with additive control and noise, and for a class of inhomogeneous semi-linear systems with nonlinear dependence on the control variable. For the first class of systems, we obtain a particularly interesting result which displays a trade-off between the volume-expansion rate of the system and the mass distribution of the probability measure coming from the AMS property. This trade-off is a specific feature of nonlinear systems, since in the linear case the influence of the measure is canceled out due to the fact that the Jacobian matrix with respect to the state is a constant in this case. From our results we can easily recover the well-known capacity bound for linear systems, $\sum_{\lambda}\max\{0,\log|\lambda|\}$ (summing over all eigenvalues of the dynamical matrix), and also previous bounds for nonlinear systems proved via information-theoretic methods.%

We emphasize that for the case of noisy channels, at least for a simple class of scalar systems, we are able to derive similar lower bounds as for noiseless channels via relating the number of control sequences needed for stabilization to a state estimation problem, and then by a generalization of the strong converse to the channel coding problem in information theory together with optimal transport theory, relating the channel capacity to a state estimation problem. This approach, in particular, allows for replacing arguments which depend on the maximum number of possible distinct message sequences for noiseless channels with an entropy-theoretic argument. It is our hope that this novel method will also be accessible to a general readership and find further applications.%

The paper is organized as follows. In Section \ref{sec_litreview} we provide a short literature review. The technical details of the stabilization problem are outlined in Section \ref{sec_problem}. The subsequent Section \ref{sec_stabent} introduces the notion of stabilization entropy. Applications to specific system models are given in Sections \ref{sec_specmodels1} and \ref{sec_specmodels2}, and Section \ref{sec_noisychannels} contains our result for noisy channels. Finally, the proofs of two technical lemmas are given in the Appendix.%

\section{A brief literature review}\label{sec_litreview}

This paper continues along the research programs developed in \cite{Kaw}, which considers deterministic systems, and \cite{yuksel2016stability}, which considers stochastic systems. For comprehensive literature reviews on the subject, we refer to \cite{Kaw,MatveevSavkin,YukselBasarBook}. Here we only provide a short review of the most relevant contributions.%

For noise-free linear systems controlled over discrete noiseless channels, various authors have obtained a formula for the smallest channel capacity above which stabilization is possible, under various assumptions on the system and the admissible coders and controllers. This result is usually referred to as a \emph{data-rate theorem} and asserts that the smallest capacity is given by the logarithm of the unstable determinant of the open-loop system, i.e., the log-sum of the unstable eigenvalues. The earliest works in this context are Wong \& Brockett \cite{Brockett} and Baillieul \cite{Baillieuil99}. More general versions of the data-rate theorem have been proven in Tatikonda \& Mitter \cite{Tatikonda} and Hespanha et al.~\cite{HOV}. For noisy systems and mean-square stabilization, or more generally, moment-stabilization, analogous data-rate theorems have been proven in Nair \& Evans \cite{NairEvans} and Sahai \& Mitter \cite{SahaiParts}, see also \cite{Martins,SavkinSIAM09}. For extensive reviews, see \cite{AMF,FMi,MatveevSavkin,Ne2,YukselBasarBook}. A data-rate theorem for AMS stability of linear systems was established in \cite[Thm.~8.5.3]{YukselBasarBook} (see also \cite[Thm.~3.1]{yuksel2016stability}) and \cite[Thm.~4.1 and 4.2]{YukselAMSITArxiv}, \cite[Thm.~2.2, 3.2 and 3.5]{AndrewJohnstonReport} under various variations. A recent study along a similar construction to the one introduced in \cite{yukselfixedrate} and \cite{YukMeynTAC2010} under fixed-rate quantization is \cite{KostinaGireejaCDC18}.%

The studies of nonlinear systems have typically considered deterministic systems that are noise-free systems controlled over discrete noiseless channels. In this context, Nair et al.~\cite{Nea} introduced the notion of \emph{topological feedback entropy} (in analogy to topological entropy for dynamical systems \cite{AKM}) for discrete-time systems to characterize the smallest average rate of information above which the state can be kept inside a compact controlled invariant set. They also characterized the smallest data rate for stabilization to an equilibrium point as the log-sum of the unstable eigenvalues of the linearization. Colonius \& Kawan in \cite{CKa} introduced the notion of \emph{invariance entropy} for continuous-time systems for the same stabilization objective. When adapted to the same (discrete-time) setting, the two notions are equivalent, see \cite{CKN}. A comprehensive review of these concepts is provided in \cite{Kaw}. We also note that recently a concept of metric invariance entropy based on conditionally invariant measures was established in \cite{Co2}. Further studies on control of nonlinear systems over communication channels have focused on constructive schemes (and not on converse theorems), primarily for noise-free systems and channels, see, e.g., \cite{baillieul2004data,de2004stabilizability,liberzon2005stabilization}.%

We also emphasize that for nonlinear systems the problems of local stabilization (stabilization to a point), semi-global stabilization (set invariance) and global stabilization (as in the stochastic stabilization criterion considered here) are fundamentally different from each other, while for linear systems they can all be handled with similar methods, leading to the above-mentioned data-rate theorem in each case. This is related to the fact that for linear systems any local (dynamical or control-theoretic) property is a global property as well. For nonlinear systems, linearization techniques work well for local problems, for semi-global problems only under specific assumptions and for global problems almost not at all. In addition, the presence of (possibly unbounded and additive) noise requires an approach fundamentally different from the machinery utilized for local stabilization problems.%

\section{Preliminaries and problem description}\label{sec_problem}

\paragraph{Notation} If $A$ is a finite set, we write $\# A$ for its cardinality. The complement of a set $A \subset X$ is denoted by $A^c = X \backslash A$. We write $\unit_A$ for the indicator function of a set $A$. By $\log$ we always denote the base-$2$-logarithm. We write $\Z_+$ for the set of nonnegative integers and put $\Z_{>0} := \Z_+ \backslash \{0\}$. Moreover, we use the notation $[a;b]$ for a discrete interval, i.e., $[a;b] = \{a,a+1,\ldots,b\}$ for any $a,b \in \Z$ with $a \leq b$. By $|\cdot|$ we denote the standard Euclidean norm on $\R^N$ and by $\|\cdot\|$ any associated operator norm. We write $B_r(x) = \{ y \in \R^N : |x - y| < r \}$ for $x\in\R^N$, $r > 0$, and denote by $\overline{A}$ the closure of a set $A \subset \R^N$. The Lebesgue measure on $\R^N$ is denoted by $m$. We write $I$ for the $N \tm N$-identity matrix and $\Gl(N,\R)$ for the general linear group of $\R^N$. By $\LC(V,W)$ we denote the space of all linear maps between vector spaces $V,W$. We use the notation $\mathrm{supp}(\mu)$ for the support of a Borel probability measure $\mu$. The expectation of a random variable $X$ is denoted by $E[X]$. The entropy of a $\{0,1\}$-valued Bernoulli random variable $X$ with $P(X=0)=r$ is denoted by $H(r)$, i.e., $H(r) = -r\log r - (1-r)\log (1-r)$. The relative entropy of two probability mass functions $p(x)$ and $q(x)$ on a discrete space $\mathbb{X}$ is defined by $D(p||q) := \sum_{x\in\mathbb{X}}p(x)\log\frac{p(x)}{q(x)}$. We refer the reader to \cite{Cover} for further information-theoretic concepts such as mutual information and channel capacity.%

If $\mu,\nu$ are two measures on the same measurable space, we write $\mu \ll_b \nu$ to denote that $\mu$ is absolutely continuous with respect to $\nu$ and its density is essentially bounded.%

If $X^{\Z_+}$ is the set of all sequences in some set $X$, we write $\bar{x} = (x_t)_{t\in\Z_+}$ for elements of $X^{\Z_+}$ and $\theta:X^{\Z_+} \rightarrow X^{\Z_+}$ for the left shift operator, i.e.,%
\begin{equation*}
  (\theta\bar{x})_t = x_{t+1} \mbox{\quad for all\ } t\in \Z_+,\ \bar{x} \in X^{\Z_+}.%
\end{equation*}
Moreover, we write $\bar{x}_{[0,t]} = (x_0,x_1,\ldots,x_t)$ for $t \in \Z_+$ and $\BC(X)$ for the Borel $\sigma$-field of a topological space $X$.%

To avoid technical problems concerning the measurability of certain sets, we make the following general assumption.%

\begin{assumption}\label{ass_stdborel}
We assume that all measurable spaces in this paper are standard Borel and all random variables associated with a given control system are modeled on a common (standard Borel) probability space $(\Omega,\FC,P)$. 
\end{assumption}

The standard Borel space assumption leads to useful universal measurability properties which are utilized in the paper. A measurable image of a Borel set is called an \emph{analytic} set \cite[App.~2]{dynkin1979controlled}. We note that this is evidently equivalent to the seemingly more restrictive condition of being a \emph{continuous} image of a Borel set. The following property will be utilized in our analysis: The image of a Borel set under a measurable map, and hence an analytic set, is universally measurable \cite{dynkin1979controlled}.%

Throughout the paper, we consider a stochastic control system%
\begin{equation}\label{eq_stochsys}
  x_{t+1} = f(x_t,u_t,w_t),\quad t = 0,1,2,\ldots%
\end{equation}
This defines a measurable map $ f:\R^N \tm U \tm W \rightarrow \R^N$, %
where $\R^N$ is endowed with the Borel $\sigma$-field $\BC(\R^N)$, $(U,\FC_U)$ is a measurable space and $(W,\FC_W,\nu)$ a probability space. The noise is modeled by an i.i.d.~sequence $(w_t)_{t\in\Z_+}$ of random variables on $(W,\FC_W)$ with associated probability measure $\nu$. The initial state $x_0$ is modeled by another random variable with probability measure $\pi_0$ on $(\R^N,\BC(\R^N))$ and is assumed to be independent of $(w_t)_{t\in\Z_+}$.%

We write $\varphi(t,x_0,\bar{u},\bar{w})$, $t\in\Z_+$, for the unique trajectory with initial value $x_0 \in \R^N$ associated with the noise realization $\bar{w} \in W^{\Z_+}$ and the control sequence $\bar{u} \in U^{\Z_+}$.%

We assume that an encoder, knowing the states $x_0,x_1,\ldots,x_t$ at time $t\in\Z_+$, transmits at time $t \in \Z_+$ a symbol $q_t$ through a noiseless discrete channel to a decoder/controller. We assume that the decoder receives the signals without delay. The finite coding alphabet is denoted by $\MC$ and the capacity of the channel is%
\begin{equation*}
  C = \log\#\MC.%
\end{equation*}
Thus, at time $t$, the controller has the symbol string $q_{[0,t]} = (q_0,q_1,\ldots,q_t) \in \MC^{t+1}$ available to generate the control input $u_t$. Any coding and control policy of this form is called a \emph{causal coding and control policy}. A more general setup including a noisy channel will be introduced and studied in Section \ref{sec_noisychannels}.%

The considered control objective is to make the state process $(x_t)_{t\in\Z_+}$ asymptotically mean stationary (AMS). Writing
${\mathbf P}$ for the process measure on $(\R^N)^{\Z_+}$, i.e.,%
\begin{equation*}
  {\mathbf P}(F) = P(\{\omega\in\Omega : (x_t(\omega))_{t\in\Z_+} \in F\}),%
\end{equation*}
the process $\{x_t\}_{t\in\Z_+}$ is AMS if there is a probability measure $\bar{P}$ on $\BC((\R^N)^{\Z_+})$ with%
\begin{equation*}
  \lim_{T\rightarrow\infty}\frac{1}{T}\sum_{t=0}^{T-1} {\mathbf P}(\theta^{-t}F) = \bar{P}(F) \mbox{\quad for all\ } F \in \BC((\R^N)^{\Z_+}).%
\end{equation*}
This implies that $\bar{P}$ is a stationary measure for $(x_t)$, i.e., $\bar{P}(\theta^{-t}F) = \bar{P}(F)$ for all times $t$ and Borel sets $F$.%

The AMS property implies the existence of a probability measure $Q$ on $(\R^N,\BC(\R^N))$ so that%
\begin{equation}\label{eq_ams}
  \lim_{T \rightarrow \infty}\frac{1}{T}\sum_{t=0}^{T-1}P(x_t \in A) = Q(A)%
\end{equation}
for every $A \in \BC(\R^N)$. This can be seen by considering sets of the form $F = A \tm \R^N \tm \R^N \tm \cdots$. Then ${\mathbf P}(\theta^{-t}F)$ reduces to $P(x_t \in A)$ and the measure $Q$ is given by $Q(A) = \bar{P}(F)$.%

We note that it was shown in \cite[Thm.~5.1]{yuksel2016stability} that an additive noise system can be made AMS over a finite-capacity channel under mild assumptions. Thus, searching for lower bounds on the necessary channel capacity is a meaningful problem.%

\section{Stabilization entropy}\label{sec_stabent}

%
\begin{definition}
For any Borel set $B \subset \R^N$, $T\in\Z_{>0}$ and $\rho,r\in(0,1)$, a set $S \subset U^T$ is called \emph{$(T,B,\rho,r)$-spanning} if there exists a set $\tilde{\Omega} \in \FC$ with $P(\tilde{\Omega}) \geq 1 - \rho$ so that for every $\omega \in \tilde{\Omega}$ there is $\bar{u} \in S$ with%
\begin{equation}\label{eq_stabent}
  \frac{1}{T} \# \left\{ t\in[0;T-1] : \varphi(t,x_0(\omega),\bar{u},\bar{w}(\omega)) \in B \right\} \geq 1 - r.%
\end{equation}
We write $s_B(T,\rho,r)$ for the smallest cardinality of a $(T,B,\rho,r)$-spanning set (where $s_B(T,\rho,r)=\infty$ if no finite $(T,B,\rho,r)$-spanning set exists) and define the \emph{$(B,\rho,r)$-stabilization entropy of system \eqref{eq_stochsys} by}%
\begin{equation*}
  h_B(\rho,r) := \limsup_{T\rightarrow\infty}\frac{1}{T}\log s_B(T,\rho,r).%
\end{equation*}
\end{definition}

Some remarks about this definition are in order:%

(i) The control sequences $\bar{u}$ in the above definition are not generated by a coding and control policy. Indeed, $h_B(\rho,r)$ is an intrinsic quantity of the open-loop system.%

(ii) The existence and finiteness of $(T,B,\rho,r)$-spanning sets is not immediately clear from the definition. However, as we will see below, in relevant cases this is guaranteed. In general, we always have $0 \leq h_B(\rho,r) \leq \infty$.%

(iii) There are some obvious monotonicity properties of the function $h_B(\cdot,\cdot)$. Namely, if $r$ or $\rho$ become smaller, $h_B(\rho,r)$ increases. This in particular implies the existence of corresponding limits as $r \rightarrow 0$ and $\rho \rightarrow 0$ (which may be infinite).%

(iv) The notion of $(B,\rho,r)$-stabilization entropy is defined in close analogy to the notion of measure-theoretic $r$-entropy \cite{Re1,Re2} for dynamical systems. This quantity generalizes the classical Kolmogorov-Sinai measure-theoretic entropy, on the basis of its characterization due to Katok \cite{Kat} for ergodic measures. While the original definition of measure-theoretic entropy is based on computing the Shannon entropy of ``dynamical partitions'', Katok's characterization is based on counting the minimal number of ``dynamical balls'' of a certain radius needed to cover a subset of the state space with measure greater than some threshold.%

We now present our key lemma which relates the channel capacity necessary for stabilization to the stabilization entropy. In particular, it shows that finite $(T,B,\rho,r)$-spanning sets exist for appropriate choices of $B,\rho,r$, provided that the AMS property can be achieved.%

\begin{lemma}\label{lem_AMS}
Assume that the AMS property is achieved via a causal coding and control policy over a noiseless channel of capacity $C$. Then for every Borel set $B \subset \R^N$ with $0 < Q(B) < 1$ and all sufficiently small $\varepsilon>0$ we have%
\begin{equation*}
  C \geq h_B\left(\frac{1 + \frac{\varepsilon}{2}}{1 + \varepsilon},(1 + \varepsilon)Q(B^c)\right).%
\end{equation*}
If $Q(B) = 1$, then for all $r\in(0,1)$ and $\varepsilon>0$ sufficiently small we have%
\begin{equation*}
  C \geq h_B\left(\frac{1 + \frac{\varepsilon}{2}}{1 + \varepsilon},(1 + \varepsilon)r\right).%
\end{equation*}
\end{lemma}

\begin{IEEEproof}
We fix a causal coding and control policy which achieves the AMS property over a noiseless channel of capacity $C$. For a given set $B \in \BC(\R^N)$ with $0<Q(B)<1$, \eqref{eq_ams} implies%
\begin{equation}\label{eq_AMS_appl}
  \lim_{T \rightarrow \infty}\frac{1}{T}\sum_{t=0}^{T-1} P(x_t \in B^c) = 1 - Q(B) =: r.%
\end{equation}
Since $P(x_t \in B^c) = E[\unit_{B^c}(x_t)]$, this can also be written as%
\begin{equation*}
  \lim_{T \rightarrow \infty}E\Bigl[\frac{1}{T}\sum_{t=0}^{T-1} \unit_{B^c}(x_t) \Bigr] = r.%
\end{equation*}
We pick $\varepsilon \in (0,(1-r)/r)$ and choose $T_0 > 0$ so that%
\begin{equation*}%
  E\Bigl[\frac{1}{T}\sum_{t=0}^{T-1} \unit_{B^c}(x_t) \Bigr] \leq \left(1 + \frac{\varepsilon}{2}\right)r,\quad \forall T \geq T_0.%
\end{equation*}
By Markov's inequality, this implies that for every $T \geq T_0$ the event%
\begin{equation*}
  \tilde{\Omega}_T := \Bigl\{ \omega \in \Omega\ :\ \frac{1}{T}\sum_{t=0}^{T-1} \unit_{B^c}(x_t(\omega)) \leq (1 + \varepsilon)r \Bigr\}%
\end{equation*}
occurs with probability $P(\tilde{\Omega}_T) \geq \varepsilon/(2(1 + \varepsilon))$, since%
\begin{align*}
  P\Bigl(\frac{1}{T}\sum_{t=0}^{T-1}\unit_{B^c}(x_t) > (1 + \varepsilon)r\Bigr) &\leq \frac{E[\frac{1}{T}\sum_{t=0}^{T-1}\unit_{B^c}(x_t)]}{(1+\varepsilon)r}\\
	&\leq \frac{1 + \frac{\varepsilon}{2}}{1 + \varepsilon} = 1 - \frac{\varepsilon}{2(1+\varepsilon)}.%
\end{align*}
Observe that for every $\omega \in \tilde{\Omega}_T$ the number of $t$'s in $[0;T-1]$ satisfying $x_t(\omega) \in B^c$ is $\leq (1 + \varepsilon)rT$. Now for every $T \geq T_0$ consider the set%
\begin{equation*}
  S_T := \bigl\{ \bar{u}_{[0,T-1]}(\omega) \in U^T : \omega \in \tilde{\Omega}_T \bigr\}%
\end{equation*}
of control sequences generated in the time interval $[0;T-1]$ provided that $\omega \in \tilde{\Omega}_T$ and $x_0 = x_0(\omega)$, $\bar{w} = \bar{w}(\omega)$. Since the maximal number of different messages that can be transmitted in the time interval $[0;T-1]$ is $(\#\MC)^T$, we have%
\begin{equation*}
  \# S_T \leq (\#\MC)^T.%
\end{equation*}
We claim that $S_T$ is $(T,B,\frac{1 + \varepsilon/2}{1 + \varepsilon},(1 + \varepsilon)r)$-spanning. Indeed,%
\begin{equation*}
  P(\tilde{\Omega}_T) \geq \frac{\varepsilon}{2(1 + \varepsilon)} = 1 - \frac{1 + \frac{\varepsilon}{2}}{1 + \varepsilon},%
\end{equation*}
for every $\omega \in \tilde{\Omega}_T$ we have $\bar{u}_{[0,T-1]}(\omega) \in S_T$, and the number of $t$'s in $[0;T-1]$ with $x_t(\omega) = \varphi(t,x_0(\omega),\bar{u}(\omega),\bar{w}(\omega)) \in B$ is $\geq T - (1+\varepsilon)rT = (1 - (1 + \varepsilon)r)T$. Hence,%
\begin{equation*}
  s_B\Bigl(T,\frac{1 + \frac{\varepsilon}{2}}{1 + \varepsilon},(1 + \varepsilon)r\Bigr) \leq \# S_T \leq (\# \MC)^T \mbox{\quad for all\ } T \geq T_0.%
\end{equation*}
Taking logarithms, dividing by $T$ and letting $T \rightarrow \infty$ yields the assertion. The case $Q(B) = 1$ is handled by replacing \eqref{eq_AMS_appl} with the inequality%
\begin{equation*}
  \lim_{T \rightarrow \infty}\frac{1}{T}\sum_{t=0}^{T-1} P(x_t \in B^c) = Q(B^c) = 0 \leq r%
\end{equation*}
for an arbitrarily chosen $r\in(0,1)$, and applying the same arguments.%
\end{IEEEproof}

Lemma \ref{lem_AMS}, while sounding technical, has significant consequences, since it allows for the application of volume-growth arguments that have been used in the literature for deterministic settings.%

\section{Volume-expanding systems}\label{sec_specmodels1}

In this section, we assume throughout that the measure $\pi_0$ of the random variable $x_0$ is absolutely continuous w.r.t.~the Lebesgue measure $m$ on $\R^N$ and that the associated density is essentially bounded, i.e., $\pi_0 \ll_b m$.%


Consider a system of the form%
\begin{equation}\label{eq_volexpsys}
  x_{t+1} = f(x_t) + u_t + w_t%
\end{equation}
with $U = W = \R^N$ and an injective $C^1$-map $f:\R^N \rightarrow \R^N$ satisfying (with $\rmD f(x)$ denoting the Jacobian of $f$ at $x$)%
\begin{equation}\label{eq_volexp}
  |\det\rmD f(x)| \geq 1 \mbox{\quad for all\ } x\in\R^N.%
\end{equation}

\begin{theorem}\label{thm_volexp}
Consider system \eqref{eq_volexpsys} satisfying \eqref{eq_volexp} and $\pi_0 \ll_b m$. Assume that the AMS property is achieved with an associated AMS measure $Q$ via a causal coding and control policy over a noiseless channel of capacity $C$. Then for all Borel sets $B\subset\R^N$ with $0<m(B)<\infty$ we have%
\begin{equation}\label{eq_expsyslb}
  C \geq Q(B)\log\inf_{x\in B}|\det\rmD f(x)|.%
\end{equation}
\end{theorem}

\begin{IEEEproof}
The proof is subdivided into four steps.%

\emph{Step 1.} Fix a Borel set $B$ with $0 < m(B) < \infty$ and let $S \subset U^T$ be a finite $(T,B,\rho,r)$-spanning set (if a finite spanning set does not exist for any $T$, the estimate becomes trivial). For the associated $\tilde{\Omega} \subset \Omega$ with $P(\tilde{\Omega}) \geq 1 - \rho$, define%
\begin{align*}
  A &:= \left\{ (\bar{w}(\omega),x_0(\omega)) : \omega \in \tilde{\Omega} \right\},\\
	A(\bar{u}) &:= \Bigl\{ (\bar{w},x) \in W^{\Z_+}\tm\R^N : \\
	    & \frac{1}{T} \# \left\{ t\in[0;T-1]\ :\ \varphi(t,x,\bar{u},\bar{w}) \in B \right\} \geq 1 - r \Bigr\},\\
  A(\bar{u},\bar{w}) &:= \{ x \in \R^N : (\bar{w},x) \in A(\bar{u}) \}%
\end{align*}
for all control and noise sequences $\bar{u}$ and $\bar{w}$, respectively. Note that the (universal) measurability of $A$ follows from Assumption \ref{ass_stdborel}. From the definition of $(T,B,\rho,r)$-spanning sets it immediately follows that%
\begin{equation}\label{eq_sets1}
  A \subset \bigcup_{\bar{u} \in S}A(\bar{u})%
\end{equation}
and we have (by Tonelli's theorem)%
\begin{equation}\label{eq_sets2}
  (\nu^{\Z_+} \tm m)(A(\bar{u})) = \int \nu^{\Z_+}(\rmd\bar{w}) m(A(\bar{u},\bar{w})).%
\end{equation}
We can write $A(\bar{u},\bar{w})$ as the disjoint union of the sets%
\begin{align*}
  A(\bar{u},\bar{w},\Lambda) &:= \{x \in \R^N : \forall t\in [0;T-1],\\
	                          &\qquad \varphi(t,x,\bar{u},\bar{w}) \in B \Leftrightarrow t \in \Lambda\},%
\end{align*}
where $\Lambda$ ranges through all subsets of $[0;T-1]$ with cardinality $\geq (1-r)T$. Then%
\begin{equation}\label{eq_sets3}
  m(A(\bar{u},\bar{w})) = \sum_{\Lambda \subset [0;T-1]\atop \#\Lambda \geq (1-r)T} m(A(\bar{u},\bar{w},\Lambda)).%
\end{equation}
Now we prove that%
\begin{equation}\label{eq_ameas}
  (\nu^{\Z_+} \tm m)(A) \geq \alpha%
\end{equation}
for a constant $\alpha>0$, independent of $T$. First, observe that by the independence of the random variables $x_0$ and $\bar{w} = (w_t)_{t\in\Z_+}$, $\nu^{\Z_+} \tm \pi_0$ is the probability measure of the joint variable $(\bar{w},x_0)$. Hence,%
\begin{align*}
  (\nu^{\Z_+} \tm \pi_0)(A) &= P(\{\omega \in \Omega : (\bar{w}(\omega),x_0(\omega)) \in A\})\\
	&\geq P(\tilde{\Omega}) \geq 1 - \rho.%
\end{align*}
If we write $p$ for the density of $\pi_0$ with respect to $m$ and assume that $p(x) \leq \gamma < \infty$, we thus find that%
\begin{align*}
  1 - \rho &\leq (\nu^{\Z_+} \tm \pi_0)(A) = \int \nu^{\Z_+}(\rmd\bar{w}) m(\rmd x) \unit_A(\bar{w},x) p(x)\\
	&\leq \gamma\, (\nu^{\Z_+} \tm m)(A),%
\end{align*}
implying that \eqref{eq_ameas} holds with the constant $\alpha := (1-\rho)/\gamma$.%

\emph{Step 2.} Writing $\varphi_{t,\bar{u},\bar{w}}(\cdot) = \varphi(t,\cdot,\bar{u},\bar{w})$, we define%
\begin{equation*}
  A_t(\bar{u},\bar{w},\Lambda) := \varphi_{t,\bar{u},\bar{w}}(A(\bar{u},\bar{w},\Lambda)),\quad t = 0,1,\ldots,T-1.%
\end{equation*}
Then we have%
\begin{equation*}
  A_t(\bar{u},\bar{w},\Lambda) \subset \left\{\begin{array}{rl}
	   B & \mbox{\quad for all\ } t \in \Lambda\\
		 B^c & \mbox{\quad for all\ } t \in [0;T-1]\backslash\Lambda%
		\end{array}\right.%
\end{equation*}
which immediately implies that for $c := \inf_{x\in B}|\det\rmD f(x)|$ (using that $f$ is injective and $C^1$)%
\begin{align*}
  m(A_{t+1}(\bar{u},\bar{w},\Lambda)) &\geq c \cdot m(A_t(\bar{u},\bar{w},\Lambda)) \mbox{\quad for all\ } t\in \Lambda,\\
  m(A_{t+1}(\bar{u},\bar{w},\Lambda)) &\geq m(A_t(\bar{u},\bar{w},\Lambda)) \mbox{\quad for all\ } t\notin \Lambda.%
\end{align*}
Let $t^* = t^*(\Lambda) := \max\Lambda$. Then an inductive argument yields%
\begin{equation}\label{eq_sets4}
  m(B) \geq m(A_{t^*}(\bar{u},\bar{w},\Lambda)) \geq c^{\#\Lambda-1} \cdot m(A(\bar{u},\bar{w},\Lambda)).%
\end{equation}

\emph{Step 3.} Combining \eqref{eq_sets4}, \eqref{eq_sets1}, \eqref{eq_sets2}, \eqref{eq_sets3} and \eqref{eq_ameas}, we obtain%
\begin{align*}
  \alpha &\leq (\nu^{\Z_+} \tm m)(A) \leq \# S \cdot \max_{\bar{u} \in S} (\nu^{\Z_+} \tm m)(A(\bar{u}))\allowdisplaybreaks\\
	&= \# S \cdot \max_{\bar{u}\in S} \int \nu^{\Z_+}(\rmd\bar{w}) m(A(\bar{u},\bar{w}))\allowdisplaybreaks\\
	&= \# S \cdot \max_{\bar{u}\in S} \sum_{\Lambda \subset [0;T-1]\atop \#\Lambda \geq (1-r)T}\int \nu^{\Z_+}(\rmd\bar{w}) m(A(\bar{u},\bar{w},\Lambda))\allowdisplaybreaks\\
	&\leq \# S \cdot \max_{\bar{u}\in S} \sum_{\Lambda \subset [0;T-1]\atop \#\Lambda \geq (1-r)T}\int \nu^{\Z_+}(\rmd\bar{w}) \frac{1}{c^{\#\Lambda-1}} m(A_{t^*(\Lambda)}(\bar{u},\bar{w},\Lambda))\allowdisplaybreaks\\
	&\leq \# S \cdot \frac{1}{c^{(1-r)T-1}} \max_{\bar{u}\in S}\\
	&\sum_{\Lambda \subset [0;T-1]\atop \#\Lambda \geq (1-r)T}\int \nu^{\Z_+}(\rmd\bar{w}) m(B \cap A_{t^*(\Lambda)}(\bar{u},\bar{w},\Lambda))\allowdisplaybreaks\\
	&= \# S \cdot \frac{1}{c^{(1-r)T-1}} \max_{\bar{u}\in S} \sum_{t=\lceil (1-r)T \rceil-1}^{T-1}\\
	&\sum_{\Lambda:\ t^*(\Lambda) = t}\int \nu^{\Z_+}(\rmd\bar{w}) m(B \cap A_t(\bar{u},\bar{w},\Lambda))\allowdisplaybreaks\\
	&= \# S \cdot \frac{1}{c^{(1-r)T-1}} \max_{\bar{u}\in S} \sum_{t=\lceil (1-r)T \rceil-1}^{T-1}\\
	&\int \nu^{\Z_+}(\rmd\bar{w}) \sum_{\Lambda:\ t^*(\Lambda) = t}m(B \cap A_t(\bar{u},\bar{w},\Lambda))\allowdisplaybreaks\\
 &\stackrel{(\diamond)}{=} \# S \cdot \frac{1}{c^{(1-r)T-1}} \max_{\bar{u}\in S} \sum_{t=\lceil (1-r)T \rceil-1}^{T-1} \\
&\int \nu^{\Z_+}(\rmd\bar{w})m\Bigl(B \cap \bigcup_{\Lambda:\ t^*(\Lambda) = t} A_t(\bar{u},\bar{w},\Lambda)\Bigr)\allowdisplaybreaks\\
&\leq \# S \cdot \frac{rT+1}{c^{(1-r)T-1}} \cdot m(B).%
\end{align*}
In $(\diamond)$ we use that the sets $A(\bar{u},\bar{w},\Lambda)$, $\Lambda \subset [0;T-1]$, are pairwise disjoint. Because of the assumption that $f$ and hence $\varphi_{t,\bar{u},\bar{w}}$ (for each $t$) is injective, this implies that also the sets $A_t(\bar{u},\bar{w},\Lambda)$ are pairwise disjoint. Hence, we can conclude that%
\begin{equation*}
  h_B(\rho,r) \geq \limsup_{T\rightarrow\infty} \frac{1}{T} \log\frac{c^{(1-r)T-1}}{rT+1} = (1-r)\log c.%
\end{equation*}

\emph{Step 4.} We complete the proof by applying Lemma \ref{lem_AMS}. Let us first assume that $0 < Q(B) < 1$. Then Lemma \ref{lem_AMS} together with Step 3 yields%
\begin{align*}
  C &\geq h_B\left(\frac{1 + \frac{\varepsilon}{2}}{1 + \varepsilon},(1 + \varepsilon)Q(B^c)\right)\\
	&\geq (1 - (1+\varepsilon)Q(B^c)) \log \inf_{x\in B}|\det\rmD f(x)|.%
\end{align*}
As $\varepsilon\rightarrow0$, the desired inequality follows. The case $Q(B) = 0$ is trivial and the case $Q(B) = 1$ follows by continuity.%
\end{IEEEproof}

\begin{remark}
The preceding theorem recovers, as a special case, \cite[Thm.~3.2]{yuksel2016stability}, which shows that $C \geq \inf_{x\in\R^N}\log|\det\rmD f(x)|$. However, the result there is more general with regard to the allowed class of channels.%
\end{remark}

\begin{remark}
In the inequality \eqref{eq_expsyslb} we see a trade-off between the $Q$-measure of the set $B$ and the infimal volume growth on $B$. If some characteristics of the measure $Q$ are known, one can try to optimize the lower bound by a careful choice of $B$. Also observe that%
\begin{equation*}
  \int Q(\rmd x)\log|\det\rmD f(x)| \geq Q(B)\inf_{x\in B}\log|\det\rmD f(x)|%
\end{equation*}
holds for all Borel sets $B$, where the left-hand side is the expected volume expansion w.r.t.~the AMS measure $Q$. Hence, it is tempting to conjecture that also the integral above is a lower bound on the capacity. Under the stronger criterion of asymptotic ergodicity, such a bound has been derived in \cite{GKY}.%
\end{remark}

The next corollary shows that imposing further properties on the AMS measure $Q$ can lead to more concrete bounds.%


\begin{corollary}\label{cor_momentcond}
Consider system \eqref{eq_volexpsys} satisfying \eqref{eq_volexp} and $\pi_0 \ll_b m$. Assume that the AMS property is achieved via a noiseless channel of capacity $C$ and the measure $Q$ satisfies for some $M,p>0$ the moment constraint%
\begin{equation*}
  \int Q(\rmd x)|x|^p \leq M.%
\end{equation*}
Then the channel capacity satisfies%
\begin{equation}\label{eq_mclb}
  C \geq \sup_{\kappa^p \geq M}\left(1 - \frac{M}{\kappa^p}\right)\min_{|x| \leq \kappa}\log|\det\rmD f(x)|.%
\end{equation}
\end{corollary}

\begin{IEEEproof}
Consider the set $B := \overline{B_{\kappa}(0)}$ for a fixed $\kappa>0$. By Markov's inequality, the moment constraint implies $Q(B) \geq 1 - \frac{M}{\kappa^p}$. Hence, Theorem \ref{thm_volexp} implies the assertion.%
\end{IEEEproof}

\begin{example}
For a linear system with $f(x) = Ax$, $A \in \R^{N \tm N}$ satisfying $|\det A| \geq 1$, our result implies the well-known relation (cf.~\cite{YukselAMSITArxiv,YukselBasarBook})%
\begin{equation*}
  C \geq \log|\det A| = \sum_{\lambda} n_{\lambda}\log|\lambda|%
\end{equation*}
with summation over all eigenvalues $\lambda$ of $A$ with associated multiplicities $n_{\lambda}$. By a simple decoupling argument this can be refined to show that $  C \geq \sum_{\lambda} \max\{0,n_{\lambda}\log|\lambda|\}$
\hfill$\diamond$
\end{example}

The next example shows that for nonlinear systems the supremum in \eqref{eq_mclb} is not necessarily attained as $\kappa \rightarrow \infty$, i.e., the lower bound \eqref{eq_expsyslb} indeed expresses a trade-off between the measure of $B$ and the minimal volume expansion on $B$.%

\begin{example}
Consider a map $f:\R \rightarrow \R$ with derivative%
\begin{equation*}
  f'(x) = \left\{ \begin{array}{rl}
	                     2 & \mbox{\ if } |x| \leq 1\\
											 2^{\frac{1}{\sqrt{|x|}}} & \mbox{\ if } |x| > 1%
									\end{array}\right.%
\end{equation*}
and note that $|f'(x)| = f'(x) > 1$ for all $x\in\R$. Since $f'$ is symmetric and monotonically decreasing on $[0,\infty)$, we obtain%
\begin{equation*}
  \min_{|x| \leq \kappa}\log|f'(x)| = \log|f'(\kappa)| \mbox{\quad for all\ } \kappa > 0.%
\end{equation*}
Corollary \ref{cor_momentcond}, applied with $M = p = 1$ thus yields the capacity bound%
\begin{equation*}
  C \geq \sup_{\kappa \geq 1} \left(1 - \frac{1}{\kappa}\right)\frac{1}{\sqrt{\kappa}}.%
\end{equation*}
A straightforward analysis shows that this supremum is attained as a maximum at $\kappa = 3$, and hence $C \geq 2/(3\sqrt{3})$.  \hfill$\diamond$%
\end{example}

\section{Inhomogeneous semilinear systems}\label{sec_specmodels2}

In this section, we also assume throughout that $\pi_0 \ll_b m$. We consider systems of the form%
\begin{equation}\label{eq_inhbil}
  x_{t+1} = A(u_t)x_t + Bv_t + w_t,%
\end{equation}
where $u_t \in U$ and $v_t \in V = \R^M$ are control variables and $w_t \in W = \R^N$ is the noise variable. We assume that $U$ is a compact, connected metric space and $A:U \rightarrow \Gl(N,\R)$ is continuous. The product space $U^{\Z}$ will be equipped with the product topology (and hence becomes a compact, connected metric space as well). Obviously, the case of linear systems with additive noise is covered here, since $A$ may be chosen to be constant.%

The homogeneous system associated with \eqref{eq_inhbil} is%
\begin{equation}\label{eq_homogen}
  x_{t+1} = A(u_t)x_t.%
\end{equation}
For a given initial state $x_0 \in \R^N$ and a control sequence $\bar{u} = (u_t)_{t\in\Z}$ we write $\Phi(t,\bar{u})x_0$ for the associated solution of \eqref{eq_homogen}. Here%
\begin{equation*}
  \Phi(t,\bar{u}) = \left\{\begin{array}{rl}
	  A(u_{t-1}) \cdots A(u_1)A(u_0) & \mbox{\ if } t \geq 1,\\
		                             I & \mbox{\ if } t = 0,\\
		A(u_t)^{-1} \cdots A(u_{-2})^{-1}A(u_{-1})^{-1} & \mbox{\ if } t < 0.%
		\end{array}\right.%
\end{equation*}
As we will see below, there always exists a finest continuous decomposition of the trivial vector bundle $U^{\Z} \tm \R^N$ into invariant subbundles:%
\begin{equation*}
  U^{\Z} \tm \R^N = \WC^1 \oplus \cdots \oplus \WC^r.%
\end{equation*}
Writing $\WC^i_{\bar{u}}$, $\bar{u}\in U^{\Z}$, for the fibers of the subbundles, their invariance can be expressed by the identities%
\begin{equation*}
  \Phi(t,\bar{u})\WC^i_{\bar{u}} = \WC^i_{\theta^t\bar{u}},\quad i = 1,\ldots,r,\quad t \in \Z,\ \bar{u} \in U^{\Z}.%
\end{equation*}
The subbundles $\WC^i$ generalize the Lyapunov spaces of a single operator, i.e., the sums of generalized eigenspaces corresponding to eigenvalues of the same modulus.%

Before we formulate our main result, we recall some facts about additive cocycles. An additive cocycle over a continuous map $T:X\rightarrow X$ is a function $\alpha:\Z_+ \tm X \rightarrow \R$, written as $(n,x) \mapsto \alpha_n(x)$, satisfying%
\begin{equation*}
  \alpha_{n+m}(x) = \alpha_n(x) + \alpha_m(T^n(x)) \mbox{\quad for all\ } n,m\in\Z_+,\ x\in X.%
\end{equation*}

\begin{lemma}\label{lem_ac}
Let $T:X\rightarrow X$ be a continuous map on a compact metric space $X$. Assume that $\alpha:\Z_+\tm X \rightarrow \R$ is a continuous additive cocycle over $T$. Then the following identities hold:%
\begin{align*}
& \inf_{x\in X}\liminf_{n\rightarrow\infty}\frac{1}{n}\alpha_n(x) = \inf_{x\in X}\limsup_{n\rightarrow\infty}\frac{1}{n}\alpha_n(x)\\
& = \lim_{n\rightarrow\infty}\frac{1}{n}\inf_{x\in X}\alpha_n(x) = \sup_{n \in \Z_{>0}}\frac{1}{n}\inf_{x\in X}\alpha_n(x).%
\end{align*}
Moreover, all infima above are attained, and the analogous identities with infima replaced by suprema hold.%
\end{lemma}

A purely topological proof of this lemma can be found in \cite[Cor.~2]{KSt}. For a proof of a more general result using ergodic theory see, e.g., \cite[App.~A]{Mor}.%

\begin{theorem}\label{thm_ibs}
Consider system \eqref{eq_inhbil}. Assume that $\pi_0 \ll_b m$ and that there exists a continuous and invariant vector bundle decomposition%
\begin{equation}\label{eq_vbdecom}
  U^{\Z} \tm \R^N = \VC^1 \oplus \VC^2%
\end{equation}
for the homogeneous system \eqref{eq_homogen}. Then, if the AMS property is achieved for \eqref{eq_inhbil} via a causal coding and control policy over a noiseless channel of capacity $C$, we have%
\begin{equation}\label{eq_inhbil_lb}
  C \geq \inf_{\bar{u}\in U^{\Z}}\limsup_{t \rightarrow \infty}\frac{1}{t}\log\left|\det\left(\Phi(t,\bar{u})_{|\VC^1_{\bar{u}}}:\VC^1_{\bar{u}} \rightarrow \VC^1_{\theta^t\bar{u}}\right)\right|.%
\end{equation}
\end{theorem}

\begin{IEEEproof}
First observe that the mapping $(t,\bar{u}) \mapsto \log|\det \Phi(t,\bar{u})_{|\VC^1_{\bar{u}}}|$ is a continuous additive cocycle over the shift $\theta:U^{\Z} \rightarrow U^{\Z}$. Hence, by Lemma \ref{lem_ac} the limit %
\begin{equation}\label{eq_volexpass}
  \lim_{t \rightarrow \infty}\frac{1}{t}\inf_{\bar{u}\in U^{\Z}}\log\left|\det\left(\Phi(t,\bar{u})_{|\VC^1_{\bar{u}}}:\VC^1_{\bar{u}} \rightarrow \VC^1_{\theta^t\bar{u}}\right)\right|%
\end{equation}
exists and coincides with the right-hand side in \eqref{eq_inhbil_lb}. If this limit is $\leq 0$, the statement becomes trivial, hence we may and will assume that it is positive.%

The proof now proceeds along the following four steps.%

\emph{Step 1.} Let us write $P(\bar{u}) \in \LC(\R^N,\R^N)$ for the projection onto $\VC^1_{\bar{u}}$ along $\VC^2_{\bar{u}}$. Observe that by the variation-of-constants formula we can write the solutions of \eqref{eq_inhbil} in the form%
\begin{equation}\label{eq_solform}
  \varphi(t,x,(\bar{u},\bar{v}),\bar{w}) = \Phi(t,\bar{u})x + \beta(t,\bar{u},\bar{v},\bar{w}).%
\end{equation}
We let $k$ denote the rank of the subbundle $\VC^1$ (i.e., the common dimension of its fibers) and write $m^k_{\bar{u}}$ for the $k$-dimensional Lebesgue measure on $\VC^1_{\bar{u}} = \im P(\bar{u})$. Observe that the invariance of $\VC^1$ and $\VC^2$ implies%
\begin{equation}\label{eq_bundle_inv}
  P(\theta^t\bar{u})\Phi(t,\bar{u}) = \Phi(t,\bar{u})P(\bar{u}),\quad \forall t \in \Z,\ \bar{u} \in U^{\Z}.%
\end{equation}
Moreover, since $\VC^1$ is a continuous subbundle, the map $\bar{u} \mapsto P(\bar{u})$ is continuous. By compactness of $U^{\Z}$, the following maximum exists:%
\begin{equation}\label{eq_defcb}
  R(b) := \max_{\bar{u} \in U^{\Z}}m^k_{\bar{u}}(P(\bar{u})\overline{B_b(0)}) < \infty.%
\end{equation}
Indeed, this follows from the fact that the Lebesgue measure of the image of a ball under a projection is proportional to the product of its non-vanishing singular values, which depend continuously on the projection.%

\emph{Step 2.} Fix $b>0$ and $\rho,r \in (0,1)$ with $r < \frac{1}{2}$. Assume that there exists a minimal finite $(T,B,\rho,r)$-spanning set $S \subset (U \tm V)^T$ for $B := \overline{B_b(0)}$ (which later will be justified by invoking Lemma \ref{lem_AMS}). Then there is $\tilde{\Omega} \subset \Omega$ with $P(\tilde{\Omega}) \geq 1 - \rho$ so that for each $\omega \in \tilde{\Omega}$ there is $(\bar{u},\bar{v}) \in S$ with%
\begin{equation*}
  \frac{1}{T} \# \left\{ t\in[0;T-1] : |\varphi(t,x_0(\omega),(\bar{u},\bar{v}),\bar{w}(\omega))| \leq b \right\} \geq 1 - r.%
\end{equation*}
Putting%
\begin{align*}
  \tilde{\Omega}(\bar{u},\bar{v}) &:= \Bigl\{\omega\in\tilde{\Omega} : \frac{1}{T} \# \{ t\in[0;T-1]\ :\\
	&\qquad |\varphi(t,x_0(\omega),(\bar{u},\bar{v}),\bar{w}(\omega))| \leq b \} \geq 1 - r \Bigr\}%
\end{align*}
for every $(\bar{u},\bar{v}) \in S$, we obtain%
\begin{equation}\label{eq_ibs_omega}
  \tilde{\Omega} = \bigcup_{(\bar{u},\bar{v}) \in S} \tilde{\Omega}(\bar{u},\bar{v}).%
\end{equation}
Using the notation $\bar{w}(\omega) = (w_t(\omega))_{t\in\Z_+}$, for any $\Lambda \subset [0;T-1]$ we define%
\begin{align*}
   Z &:= \left\{ (\bar{w}(\omega),x_0(\omega)) \in W^{\Z_+} \tm \R^N\ :\ \omega \in \tilde{\Omega} \right\},\\
	Z(\bar{u},\bar{v},\Lambda) &:= \bigl\{ (\bar{w},x) \in W^{\Z_+} \tm \R^N :\\
	&\qquad |\varphi(t,x,(\bar{u},\bar{v}),\bar{w})| \leq b, \forall t \in \Lambda \bigr\}.%
\end{align*}
Then, as in the proof of Theorem \ref{thm_volexp}, we obtain%
\begin{align}\label{eq_inb_incl}
\begin{split}
  Z &\subset \bigcup_{(\bar{u},\bar{v}) \in S}\bigcup_{\Lambda \subset [0;T-1]\atop\#\Lambda\geq(1-r)T}Z(\bar{u},\bar{v},\Lambda) \\
	&= \bigcup_{\Lambda \subset [0;T-1]\atop\#\Lambda\geq(1-r)T}\bigcup_{(\bar{u},\bar{v}) \in S}Z(\bar{u},\bar{v},\Lambda).%
\end{split}
\end{align}
We define the probability measure $\mu := \nu^{\Z_+}$ on $W^{\Z_+}$. Then, for any $\Lambda \subset [0;T-1]$,%
\begin{align*}
  & \sum_{(\bar{u},\bar{v}) \in S}\mu \tm m^k_{\bar{u}} (\id \tm P(\bar{u})(Z(\bar{u},\bar{v},\Lambda))\\
	&\leq \# S \cdot \max_{(\bar{u},\bar{v}) \in S} \mu \tm m^k_{\bar{u}} (\id \tm P(\bar{u})(Z(\bar{u},\bar{v},\Lambda))\\
	&= \# S \cdot \max_{(\bar{u},\bar{v}) \in S} \int \mu(\rmd\bar{w}) m^k_{\bar{u}}(P(\bar{u})Z(\bar{u},\bar{v},\bar{w},\Lambda))%
\end{align*}
with $Z(\bar{u},\bar{v},\bar{w},\Lambda) := \{x\in\R^N : |\varphi(t,x,(\bar{u},\bar{v}),\bar{w})| \leq b,\ \forall t\in\Lambda \}$. Fixing $\Lambda \subset [0;T-1]$ and putting $t_+ = t_+(\Lambda) := \max\Lambda$, an easy computation using \eqref{eq_solform} and \eqref{eq_bundle_inv} leads to%
\begin{align*}
  &\Phi(t_+,\bar{u})P(\bar{u})Z(\bar{u},\bar{v},\bar{w},\Lambda)\\
	&\subset P(\theta^{t_+}\bar{u})B_b(0) - P(\theta^{t_+}\bar{u})\beta(t_+,\bar{u},\bar{v},\bar{w}),%
\end{align*}
which implies (using \eqref{eq_defcb})%
\begin{align*}
 & m^k_{\theta^{t_+}\bar{u}}\left( \Phi(t_+,\bar{u})P(\bar{u})Z(\bar{u},\bar{v},\bar{w},\Lambda) \right)\\
 &\qquad \leq m^k_{\theta^{t_+}\bar{u}}\left(P(\theta^{t_+}\bar{u})B_b(0)\right) \leq R(b).%
\end{align*}
Now%
\begin{align*}
  &m^k_{\theta^{t_+}\bar{u}}\left(\Phi(t_+,\bar{u})P(\bar{u})Z(\bar{u},\bar{v},\bar{w},\Lambda)\right)\\
	&= \left|\det \Phi(t_+,\bar{u})_{|\VC^1_{\bar{u}}}\right| \cdot m^k_{\bar{u}}\left(P(\bar{u})Z(\bar{u},\bar{v},\bar{w},\Lambda)\right).%
\end{align*}
Putting everything together, we end up with%
\begin{align*}
  & \sum_{(\bar{u},\bar{v}) \in S}\mu \tm m^k_{\bar{u}}\left(\id \tm P(\bar{u})(Z(\bar{u},\bar{v},\Lambda))\right)\\
	&\leq \# S \cdot \max_{(\bar{u},\bar{v}) \in S} \int \mu(\rmd\bar{w}) \frac{R(b)}{|\det \Phi(t_+(\Lambda),\bar{u})_{|\VC^1_{\bar{u}}}|}\\
	&\leq \# S \cdot \sup_{\bar{u} \in U^{\Z}} \frac{R(b)}{|\det \Phi(t_+(\Lambda),\bar{u})_{|\VC^1_{\bar{u}}}|}.%
\end{align*}
To complete the proof, we have to find a reasonable lower bound for the first term above.%

\emph{Step 3.} Fix a subset $\Lambda \subset [0;T-1]$ with $\#\Lambda \geq (1-r)T$ and define $t_- = t_-(\Lambda) := \min\Lambda$. Then%
\begin{equation*}
  Z(\bar{u},\bar{v},\bar{w},\Lambda) \subset \varphi_{t_-,\bar{u},\bar{v},\bar{w}}^{-1}(\overline{B_b(0)}).%
\end{equation*}
The set on the right-hand side is contained in the closed ball%
\begin{equation*}
  \hat{B} := \overline{B_{\|\Phi(t_-,\bar{u})^{-1}\|b}(-\Phi(t_-,\bar{u})^{-1}\beta(t_-,\bar{u},\bar{v},\bar{w}))}.%
\end{equation*}
As a consequence,%
\begin{equation*}
  m(Z(\bar{u},\bar{v},\bar{w},\Lambda)) \leq m\bigl(P(\bar{u})^{-1}(P(\bar{u})Z(\bar{u},\bar{v},\bar{w},\Lambda)) \cap \hat{B}\bigr).%
\end{equation*}
Let $\langle\cdot,\cdot\rangle_{\bar{u}}$ be an inner product on $\R^N$ in which $\VC^1_{\bar{u}}$ and $\VC^2_{\bar{u}}$ are orthogonal and write $m_{\bar{u}}$ for the associated Lebesgue measure. Using compactness of $U^{\Z}$, we can do this in such a way that $m(E) \leq K\cdot m_{\bar{u}}(E)$ with a constant $K>0$ for every Lebesgue measurable set $E \subset \R^N$ and every $\bar{u} \in U^{\Z}$. For any measurable set $A \subset \VC^1_{\bar{u}}$, a simple computation yields%
\begin{equation*}
  m_{\bar{u}}(P(\bar{u})^{-1}(A) \cap \hat{B}) \leq m^k_{\bar{u}}(A) \cdot m^{N-k}((I-P(\bar{u}))\hat{B}),%
\end{equation*}
where $m^{N-k}$ denotes the $(N-k)$-dimensional Lebesgue measure on $\VC^2_{\bar{u}}$. Using again the compactness of $U^{\Z}$, we can find another constant $\tilde{K}>0$ (see Step 1) with%
\begin{equation*}
  m^{N-k}((I-P(\bar{u}))\hat{B}) \leq \tilde{K}\left(\|\Phi(t_-,\bar{u})^{-1}\|b\right)^{N-k}.%
\end{equation*}
Putting everything together, we arrive at%
\begin{align*}
  &m\left(Z(\bar{u},\bar{v},\bar{w},\Lambda)\right)\\
	&\leq \mbox{const} \cdot \|\Phi(t_-,\bar{u})^{-1}\|^{N-k} m^k_{\bar{u}}\left(P(\bar{u})Z(\bar{u},\bar{v},\bar{w},\Lambda)\right).%
\end{align*}

\emph{Step 4.} We combine the results of steps 2 and 3 to obtain%
\begin{align*}
  & \frac{\mbox{const}}{\|\Phi(t_-(\Lambda),\bar{u})^{-1}\|^{N-k}} \cdot \sum_{(\bar{u},\bar{v}) \in S} \mu \tm m(Z(\bar{u},\bar{v},\Lambda))\\
	&\leq \sum_{(\bar{u},\bar{v}) \in S}\mu \tm m^k_{\bar{u}}\left(\id \tm P(\bar{u})(Z(\bar{u},\bar{v},\Lambda))\right)\\
	&\leq \# S \cdot \sup_{\bar{u} \in U^{\Z}} \frac{R(b)}{|\det \Phi(t_+(\Lambda),\bar{u})_{|\VC^1_{\bar{u}}}|}.%
\end{align*}
Letting $n_r(T)$ denote the number of subsets of $[0;T-1]$ with $\#\Lambda \geq (1-r)T$, using \eqref{eq_inb_incl}, we end up with%
\begin{align*}
  &\gamma \leq (\mu \tm m)(Z)\\
	&\leq \mbox{const} \cdot n_r(T) \cdot \# S \cdot \max_{\Lambda \subset [0;T-1]\atop \# \Lambda \geq (1-r)T} \sup_{\bar{u} \in U^{\Z}} \frac{\|\Phi(t_-(\Lambda),\bar{u})^{-1}\|^{N-k}}{|\det \Phi(t_+(\Lambda),\bar{u})_{|\VC^1_{\bar{u}}}|}%
\end{align*}
with a positive constant $\gamma$, where the first inequality follows from $\pi_0 \ll_b m$, as in the proof of Theorem \ref{thm_volexp}. Applying the logarithm, dividing by $T$ and letting $T \rightarrow \infty$ yields%
\begin{align*}
  &0 \leq H(r) + h_B(\rho,r)\\
	&+ \limsup_{T\rightarrow\infty}\frac{1}{T} \max_{\Lambda \subset [0;T-1]\atop \# \Lambda \geq (1-r)T} \sup_{\bar{u} \in U^{\Z}} \log  \frac{\|\Phi(t_-(\Lambda),\bar{u})^{-1}\|^{N-k}}{|\det \Phi(t_+(\Lambda),\bar{u})_{|\VC^1_{\bar{u}}}|}.%
\end{align*}
Here we use, in particular, Lemma \ref{lem_largedev}. Observing that $t_-(\Lambda) \leq rT$, we can estimate%
\begin{equation*}
  \|\Phi(t_-(\Lambda),\bar{u})^{-1}\| \leq \left(\max\left\{1,\max_{u\in U}\|A(u)^{-1}\|\right\}\right)^{\lceil rT \rceil} =: c^{\lceil rT \rceil},%
\end{equation*}
leading to%
\begin{align*}
  0 &\leq H(r) + h_B(\rho,r) + (N-k) r \log c \\
	&- \liminf_{T\rightarrow\infty}\frac{1}{T}\inf_{\bar{u},\Lambda} \log\left|\det \Phi(t_+(\Lambda),\bar{u})_{|\VC^1_{\bar{u}}}\right|.%
\end{align*}
Now we use that $t_+(\Lambda) \geq (1-r)T$. Let $\alpha>0$ denote the limit in \eqref{eq_volexpass} and let $\varepsilon \in (0,\alpha)$. Then, for sufficiently large $T$,%
\begin{equation*}
  \inf_{\bar{u}\in U^{\Z}}\left|\det\Phi(t_+(\Lambda),\bar{u})_{|\VC^1(\bar{u})}\right| \geq 2^{(\alpha-\varepsilon)t_+(\Lambda)} \geq 2^{(\alpha-\varepsilon)(1-r)T}.%
\end{equation*}
Since this holds for all $\Lambda$ with $\#\Lambda \geq (1-r)T$ and $\varepsilon>0$ was arbitrary, we find that%
\begin{equation*}
  h_B(\rho,r) \geq -H(r) - (N-k) r \log c + \alpha(1 - r).%
\end{equation*}
Observe that this holds for arbitrary $b>0$, $\rho \in (0,1)$, $r \in (0,\frac{1}{2})$ and $B = \overline{B_b(0)}$. If $b$ is chosen so that $0 < Q(\overline{B_b(0)}) < 1$, Lemma \ref{lem_AMS} yields%
\begin{equation*}
  C \geq -H(r_b) - (N-k)r_b \log c + \alpha (1 - r_b),\quad r_b := Q(\overline{B_b(0)}^c).%
\end{equation*}
If $Q(\overline{B_b(0)}) < 1$ for all $b>0$, we can let $b \rightarrow \infty$, which implies $r_b \rightarrow 0$ and thus%
\begin{equation}\label{eq_C_lb_alpha}
  C \geq \alpha = \lim_{t \rightarrow \infty}\frac{1}{t}\log \inf_{\bar{u}\in U^{\Z}}\left|\det\left(\Phi(t,\bar{u})_{|\VC^1_{\bar{u}}}:\VC^1_{\bar{u}} \rightarrow \VC^1_{\theta^t\bar{u}}\right)\right|.%
\end{equation}
Otherwise, we have $Q(\overline{B_b(0)}) = 1$ for all sufficiently large $b$ and Lemma \ref{lem_AMS} yields%
\begin{equation*}
  C \geq -H(r) - (N-k)r \log c + \alpha (1 - r) \quad \forall r \in (0,1),%
\end{equation*}
also leading to \eqref{eq_C_lb_alpha}. Since $(t,\bar{u}) \mapsto \log|\det\Phi(t,\bar{u})_{|\VC^1_{\bar{u}}}|$ is a continuous additive cocycle over the shift $\theta:U^{\Z} \rightarrow U^{\Z}$, Lemma \ref{lem_ac} guarantees that the limit and the infimum in \eqref{eq_C_lb_alpha} can be interchanged (replacing $\lim$ with $\limsup$ or $\liminf$), which completes the proof.%
\end{IEEEproof}

\begin{remark}
The proof of the above theorem is partly modeled according to \cite[Thm.~3.3]{Kaw}. For a more detailed explanation of the arguments used in Step 3, see \cite[Lem.~3.3]{Kaw}.%
\end{remark}

\begin{example}
Consider the special case of a linear system, i.e., $A(u) \equiv A \in \R^{N \tm N}$. Then the vector bundle decomposition \eqref{eq_vbdecom} can be chosen as%
\begin{equation*}
  U^{\Z} \tm \R^N = (U^{\Z} \tm E^u(A)) \oplus (U^{\Z} \tm E^{cs}(A)),%
\end{equation*}
where $E^u(A)$ and $E^{cs}(A)$ are the unstable and center-stable subspace of $A$, respectively. This immediately implies%
\begin{equation*}
  C \geq \sum_{\lambda} \max\{0,n_{\lambda}\log|\lambda|\}%
\end{equation*}
with summation over the eigenvalues $\lambda$ of $A$ with algebraic multiplicities $n_{\lambda}$.  \hfill$\diamond$%
\end{example}

In the following, we will show that there always exists a finest continuous decomposition of $U^{\Z} \tm \R^N$ into invariant subbundles%
\begin{equation}\label{eq_selgrade_dec}
  U^{\Z} \tm \R^N = \WC^1 \oplus \cdots \oplus \WC^r,%
\end{equation}
which is related to the dynamical behavior of the system induced by \eqref{eq_homogen} on the projective bundle $U^{\Z} \tm \P^{N-1}$. This follows from a general result about linear flows on vector bundles known as \emph{Selgrade's theorem}, which reads as follows.%

\begin{proposition}\label{prop_selgrade}
Let $V \rightarrow B$ be a finite-dimensional real vector bundle with compact metric base space $B$. Assume that $\phi_t:V\rightarrow V$, $t\in\Z$, is a continuous discrete-time linear flow on $V$ and that the induced flow on $B$ is chain transitive. Then there exists a unique finest Morse decomposition $\MC_1,\ldots,\MC_r$ of the induced flow on the projective bundle $\P V \rightarrow B$, and $1 \leq r \leq d = \dim V_b$, $b\in B$. Every Morse set $\MC_i$ defines a $\phi_t$-invariant subbundle of $V$ via%
\begin{equation*}
  V_i = \P^{-1}\MC_i = \{v \in V\ :\ v\neq 0 \mbox{ implies } \P v \in \MC_i\}%
\end{equation*}
and the following decomposition into a Whitney sum holds:%
\begin{equation*}
  V = V_1 \oplus \cdots \oplus V_r.%
\end{equation*}
\end{proposition}

For an introduction to the concepts of chain transitivity and Morse decompositions used in this proposition we refer to \cite{CKl,PSM}. A continuous-time version of the proposition can also be found in \cite{CKl}. The discrete-time version follows from a more general result, see \cite[Thm.~6.2 and Thm.~7.5]{PSM}.%

The next proposition shows that Selgrade's theorem can be applied to the linear flow generated by equation \eqref{eq_homogen} on the trivial vector bundle $U^{\Z} \tm \R^N$.%

\begin{proposition}
The solutions of the homogeneous equation \eqref{eq_homogen} define a continuous discrete-time linear flow on the trivial vector bundle $V := U^{\Z} \tm \R^N$ with compact metric base space $U^{\Z}$. This flow is given by $\phi_t(\bar{u},x) = (\theta^t\bar{u},\Phi(t,\bar{u})x)$, $t\in\Z$. Moreover, the shift map $\theta:U^{\Z} \rightarrow U^{\Z}$ is chain transitive.%
\end{proposition}

\begin{IEEEproof}
We know that $U^{\Z}$, equipped with the product topology, is a compact and connected metric space. The flow properties ($\phi_0(\bar{u},x) = (\bar{u},x)$ and $\phi_{t+s}(\bar{u},x) = \phi_t(\phi_s(\bar{u},x))$) are easy to see. Continuity and (fiber-wise) linearity of $\phi$ are clear. From the fact that the periodic points of $\theta$ (which are precisely the periodic sequences) are dense in $U^{\Z}$, it follows that every point in $U^{\Z}$ is chain recurrent. It is well-known that a homeomorphism is chain transitive on any closed set which is connected and consists of chain recurrent points.%
\end{IEEEproof}

Combining Selgrade's theorem with Theorem \ref{thm_ibs}, we obtain the following corollary.%

\begin{corollary}\label{cor_selgrade_lb}
Consider system \eqref{eq_inhbil} and the Selgrade decomposition \eqref{eq_selgrade_dec} associated with the homogeneous system \eqref{eq_homogen}. Assume that the subbundles are ordered such that%
\begin{equation*}
  \lim_{t \rightarrow \infty}\frac{1}{t}\inf_{\bar{u}\in U^{\Z}}\log\left|\det\left(\Phi(t,\bar{u})_{|\WC^i_{\bar{u}}}:\WC^i_{\bar{u}} \rightarrow \WC^i_{\theta^t\bar{u}}\right)\right| > 0%
\end{equation*}
for $i = 1,\ldots,s$, where $s \in \{0,1,\ldots,r\}$ is the maximal number with this property. Then, if $\pi_0 \ll_b m$ and the AMS property is achieved over a noiseless channel of capacity $C$,%
\begin{equation*}
  C \geq \sum_{i=1}^s \inf_{\bar{u}\in U^{\Z}}\limsup_{t \rightarrow \infty}\frac{1}{t}\log\left|\det\left(\Phi(t,\bar{u})_{|\WC^i_{\bar{u}}}:\WC^i_{\bar{u}} \rightarrow \WC^i_{\theta^t\bar{u}}\right)\right|,%
\end{equation*}
where the right-hand side is defined as zero if $s = 0$.%
\end{corollary}

\begin{IEEEproof}
Define $\VC^1 := \WC^1 \oplus \cdots \oplus \WC^s$, $\VC^2 := \WC^{s+1} \oplus \cdots \oplus \WC^r$. Then $U^{\Z} \tm \R^N = \VC^1 \oplus \VC^2$. Since $|\det\Phi(t,\bar{u})_{|\VC^1_{\bar{u}}}|$ is, up to some multiplicative constant, the product of the numbers $|\det\Phi(t,\bar{u})_{|\WC^i_{\bar{u}}}|$, $i=1,\ldots,s$, it follows that%
\begin{align*}
   &\inf_{\bar{u}\in U^{\Z}}\limsup_{t \rightarrow \infty}\frac{1}{t}\log|\det\Phi(t,\bar{u})_{|\VC^1_{\bar{u}}}|\\
	&= \lim_{t \rightarrow \infty} \frac{1}{t} \inf_{\bar{u}\in U^{\Z}}\sum_{i=1}^s \log\left|\det\Phi(t,\bar{u})_{|\WC^i_{\bar{u}}}\right|\\
	&\geq \lim_{t \rightarrow \infty} \sum_{i=1}^s \frac{1}{t} \inf_{\bar{u}\in U^{\Z}} \log\left|\det\Phi(t,\bar{u})_{|\WC^i_{\bar{u}}}\right|\\
	&= \sum_{i=1}^s \lim_{t \rightarrow \infty} \frac{1}{t} \inf_{\bar{u}\in U^{\Z}} \log\left|\det\Phi(t,\bar{u})_{|\WC^i_{\bar{u}}}\right|\\
	&= \sum_{i=1}^s  \inf_{\bar{u}\in U^{\Z}} \limsup_{t \rightarrow \infty} \frac{1}{t} \log\left|\det\Phi(t,\bar{u})_{|\WC^i_{\bar{u}}}\right|,%
\end{align*}
where we use Lemma \ref{lem_ac} twice. This implies the result.%
\end{IEEEproof}

\begin{example}
In the special case when $r = 1$ (only one Selgrade bundle) and the system is asymptotically volume-expanding, i.e.,%
\begin{equation*}
  \lim_{t \rightarrow \infty}\frac{1}{t}\inf_{\bar{u}\in U^{\Z}}\log\left|\det\Phi(t,\bar{u})\right| > 0,%
\end{equation*}
the lower bound of Corollary \ref{cor_selgrade_lb} reduces to $C \geq \min_{u \in U}\log|\det A(u)|$. Indeed, it is easy to see that the infimum over $\bar{u}\in U^{\Z}$ is then attained at the constant sequence with value $u_* = \mathrm{argmin} |\det A(u)|$.  \hfill$\diamond$%
\end{example}

For the general case, one can use numerical methods to approximate the Lyapunov exponents, and hence, the associated volume growth rates, for the homogeneous semilinear system \eqref{eq_homogen}. For continuous-time bilinear control systems, methods for the computation of Lyapunov exponents based on algorithms for solving discounted optimal control problems have been developed in \cite{Gruene} (see also \cite[App.~D]{CKl}). In general, these methods also work for discrete-time systems.%

\section{The noisy channel case}\label{sec_noisychannels}

For discrete noiseless channels, the key idea combining the volume-growth based approaches for deterministic models with the stochastic system setup was the observation that the number of control sequences is bounded from above by the total number of received messages. This approach clearly does not directly apply to a noisy channel setup, for there can be an arbitrarily large number of possibly distinct received channel outputs, but these may not carry reliable information. In the following, we develop a new method to address this for a discrete memoryless channel (DMC). For a review of channel capacity with feedback see \cite{CsiszarKorner}, \cite[Sec.~5.3.4]{YukselBasarBook}.%

\begin{figure}[h]
\begin{center}
\includegraphics[height=2.5cm,width=8.0cm]{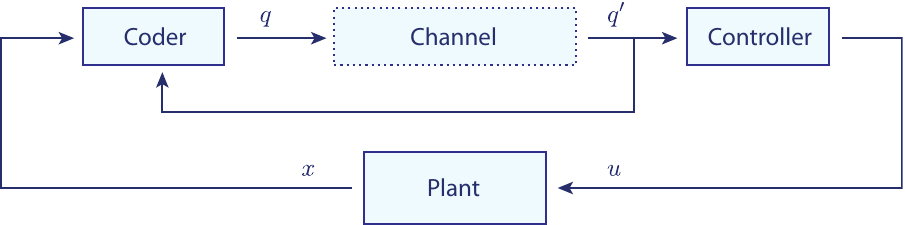}
\caption{Control of a system over a noisy channel with feedback\label{LLL1}}
\end{center}
\end{figure}

Figure \ref{LLL1} shows the control loop, using a DMC with feedback for data transmission from the encoder to the controller. The channel has a finite input alphabet $\MC$ and a finite output alphabet $\MC'$. The channel input $q_t$ at time $t$ is generated by a function $\gamma^e_t$ so that $q_t = \gamma^e_t(x_{[0,t]},q'_{[0,t-1]})$. The channel maps $q_t$ to $q'_t$ in a stochastic fashion so that $P(q'_t \in \cdot |q_t,q_{[0,t-1]},q'_{[0,t-1]}) = P(q'_t \in \cdot |q_t)$ is a conditional probability measure on $\MC'$ for all $t \in \Z_+$, for every realization $q_t,q_{[0,t-1]},q'_{[0,t-1]}$. The controller, upon receiving the information from the channel, generates its decision at time $t$, also causally: $u_t = \gamma_t^c(q'_{[0,t]})$.%

Consider a DMC with channel capacity $C$ (we note that for DMCs, it is a well-known result that feedback cannot increase the capacity). Then the following property, known as the \emph{strong converse}, holds, see \cite{Kemperman}, \cite[Problem 10.17]{CsiszarKorner}: For any $R > C$, under any coding policy:%
\begin{equation}\label{eq_StrongConverse}
  \lim_{T \rightarrow \infty}p_e(T) = 1,%
\end{equation}
where $p_e(T)$ is the average probability of error among $2^{RT}$ equally likely messages after the channel is used $T$ times under coding and decoding policies admissible according to the standard information-theoretic formulation of communication with noiseless feedback, cf.~\cite{ShannonDMCFeedback}.%

Now we consider a scalar system of the form%
\begin{equation}\label{eq_scalarsys}
  x_{t+1} = f(x_t) + u_t + w_t%
\end{equation}
with a $C^1$-function $f:\R \rightarrow \R$ satisfying%
\begin{equation}\label{eq_expanUnif}
  |f'(x)| \geq 1 \mbox{\quad for all\ } x\in\R.%
\end{equation}

Our main result reads as follows.%

\begin{theorem}\label{thm_noisychannel}
Consider system \eqref{eq_scalarsys} satisfying \eqref{eq_expanUnif}. Assume that $\pi_0 \ll m$ with $p$ denoting the density with respect to $m$, that $K := \mathrm{supp}(\pi_0)$ is a compact interval and%
\begin{equation*}
  p_{\min} := \essinf_{x \in K}p(x) > 0, \quad \quad  p_{\max} := \esssup_{x \in K}p(x) < \infty.%
\end{equation*}
Then, if the AMS property is achieved via a causal coding and control strategy over a DMC of capacity $C$, we have%
\begin{equation*}
  C \geq \inf_{x\in\R}\log|f'(x)|.%
\end{equation*}
\end{theorem}

Before the proof, it may be instructive to explain the proof approach which builds on the construction of an auxiliary coding problem that relates the number -per time stage- of distinct control actions (in a similar spirit that was the basis of the definition of stabilization entropy) to an information transmission problem and in turn to an analysis on channel capacity with feedback; by considering the fact that the number of informative messages per time stage to be transmitted with regard to the initial state cannot be less than the desired bound. The coding problem is related to a channel coding theorem via optimal transport inequalities.%

\begin{IEEEproof}
Throughout the proof, we use the following notation: Observing that we have three sources of stochasticity -- the initial state $x_0$, the noise sequence $(w_t)$ and the channel noise -- every time we make a statement about the probability $P(E)$ of an event $E$, we will add subscripts to the letter $P$, indicating which probability measures are involved in computing this probability: The subscript ``$\mathrm{i}$'' is used for the initial state, subscript ``$\mathrm{n}$'' for the noise and subscript ``$\mathrm{c}$'' for the channel.%

Let $c := \inf_{x\in\R}|f'(x)|$. Without loss of generality, we can assume that $c>1$. We prove the theorem by contradiction, assuming that $C < \log c$. First, we fix a sufficiently small $r^*>0$ so that%
\begin{equation}\label{eq_contrad_impl}
  (1 - 3r^*)\log c > C.%
\end{equation}
Since the AMS measure $Q$ is a probability measure, we can choose for every sufficiently small $\alpha \in (0,r^*)$ a $b>0$ with%
\begin{equation}\label{eq_intervalmeas}
  Q([-b,b]) > 1 - \alpha.%
\end{equation}
Later we will consider an auxiliary coding scheme, where the initial state $x_0$ is to be estimated at each time stage $T \in \Z_+$ through the knowledge of the control sequence $\bar{u} \in U^{T+1}$, applied by the controller in $[0;T]$. Given a noise realization $\bar{w}$ (that we will fix later), as an estimate for $x_0$ at time $T$ we use the center $\hat{x}_0(T,\bar{u},\bar{w})$ of the compact set%
\begin{align*}
  A_T(\bar{u},\bar{w}) &:= \Bigl\{ x\in\R\ :\ \frac{1}{T}\# \{ t\in [0;T-1] :\\
	&\qquad |\varphi(t,x,\bar{u},\bar{w})| \leq b \} \geq 1 - r^* \Bigr\},%
\end{align*}
i.e., the midpoint of $[\min A_T(\bar{u},\bar{w}),\max A_T(\bar{u},\bar{w})]$. To derive an estimate for the diameter of $A_T(\bar{u},\bar{w})$, let $x_1,x_2 \in A_T(\bar{u},\bar{w})$ be chosen arbitrarily. We claim that there exists a time $t_*$ with
$\lceil (1 - 3r^*)T \rceil \leq t_* \leq T-1$ such that%
\begin{equation*}
  \varphi(t_*,x_i,\bar{u},\bar{w}) \in [-b,b],\quad i = 1,2.%
\end{equation*}
Indeed, if this was not the case, then the number of $t$'s in the interval $[\lceil (1 - 3r^*)T \rceil;T-1]$ with $\varphi(t,x_i,\bar{u},\bar{w}) \in [-b,b]$ for each $i=1,2$ can be at most half of the cardinality of this interval, implying that the total number of $t$'s in $[0;T-1]$ such that $\varphi(t,x_i,\bar{u},\bar{w}) \in [-b,b]$ is bounded by%
\begin{align*}
&  \lceil (1 - 3r^*)T \rceil + \frac{1}{2}(T - \lceil (1 - 3r^*)T \rceil)\\
& \leq \frac{1}{2}((1 - 3r^*)T + 1) + \frac{1}{2}T\\
	&= \frac{1}{2} + \left(1 - \frac{3}{2}r^*\right)T < (1 - r^*)T,%
\end{align*}
for $T$ large enough, a contradiction. We thus obtain%
\begin{equation*}
  |x_1 - x_2| \leq \frac{2b}{c^{t_*}} \leq \frac{2b}{c^{(1-3r^*)T}},%
\end{equation*}
implying%
\begin{equation}\label{eq_midpoint}
  |x - \hat{x}_0(T,\bar{u},\bar{w})| \leq \frac{b}{c^{(1-3r^*)T}}  \mbox{\quad for all\ } x \in A_T(\bar{u},\bar{w}).%
\end{equation}
Now the AMS property together with \eqref{eq_intervalmeas} implies%
\begin{equation}\label{eq_prob1}
  \limsup_{T\rightarrow\infty} P_{\mathrm{i},\mathrm{n},\mathrm{c}}\Bigl(\frac{1}{T}\sum_{t=0}^{T-1} \unit_{[-b,b]}(x_t) < 1 - r^* \Bigr) < \frac{\alpha}{r^*}.%
\end{equation}
Indeed, this follows by an application of Markov's inequality:%
\begin{align*}
 & P_{\mathrm{i},\mathrm{n},\mathrm{c}}\Bigl(\frac{1}{T}\sum_{t=0}^{T-1}\unit_{[-b,b]}(x_t) < 1 - r^* \Bigr)\\
&= P_{\mathrm{i},\mathrm{n},\mathrm{c}}\Bigl(\frac{1}{T}\sum_{t=0}^{T-1}\unit_{[-b,b]^c}(x_t) > r^* \Bigr)\\
	&\leq \frac{1}{r^*}E_{\mathrm{i},\mathrm{n},\mathrm{c}}\Bigl[\frac{1}{T}\sum_{t=0}^{T-1}\unit_{[-b,b]^c}(x_t)\Bigr] \stackrel{T \rightarrow \infty}{\longrightarrow} \frac{Q([-b,b]^c)}{r^*} < \frac{\alpha}{r^*}.%
\end{align*}
From \eqref{eq_midpoint} and the definition of $A_T(\bar{u},\bar{w})$ we conclude that%
\begin{align*}
  & P_{\mathrm{i},\mathrm{n},\mathrm{c}}\Bigl( \frac{1}{T}\sum_{t=0}^{T-1}\unit_{[-b,b]}(x_t) < 1 - r^* \Bigr)\\
	&\geq P_{\mathrm{i},\mathrm{n},\mathrm{c}}\Bigl( |x_0 - \hat{x}_0(T,\bar{u},\bar{w})| > \frac{b}{c^{(1-3r^*)T}}\Bigr)%
\end{align*}
and the left-hand side is smaller than $\alpha/r^*$ for large $T$ by \eqref{eq_prob1}. Our aim is to show that%
\begin{equation}\label{eq_671}
  \limsup_{T \rightarrow \infty}P_{\mathrm{i},\mathrm{n},\mathrm{c}}\Bigl( |x_0 - \hat{x}_0(T,\bar{u},\bar{w})| > \frac{b}{c^{(1-3r^*)T}} \Bigr) \geq \frac{\alpha}{r^*},%
\end{equation}
leading to a contradiction with \eqref{eq_prob1}.%

To this end, we will distinguish between two complementary cases. To classify these cases, we introduce the notion of a \emph{control rate} $R$ as follows.%

For each $T\geq1$, let $\UC_T$ be the set of all possible control sequences in $U^T$ the controller can generate under the given coding and control policy, i.e.,%
\begin{align*}
  \UC_T &:= \Bigl\{ \bigl(\gamma^c_0(q'_0),\gamma^c_1(q'_{[0,1]}),\ldots,\gamma^c_{T-1}(q'_{[0,T-1]})\bigr) \in U^T:\\
	&\qquad  q'_{[0,T-1]} \in (\MC')^T \Bigr\}.%
\end{align*}
We define the control rate by%
\begin{equation*}
  R := \limsup_{T \rightarrow \infty} \frac{1}{T}\log\#\UC_T.%
\end{equation*}
We now treat the two possible cases $R < (1 - 3r^*)\log c$ and $R \geq (1 - 3r^*)\log c$ separately.%

{\bf Case 1:} We fix a noise realization $\bar{w}_*$ and prove \eqref{eq_671} for the conditional probability of the corresponding event given $\bar{w} = \bar{w}_*$. To simplify notation, we write $A_T(\bar{u})$ and $\hat{x}_0(T,\bar{u})$ instead of $A_T(\bar{u},\bar{w}_*)$ and $\hat{x}_0(T,\bar{u},\bar{w}_*)$, respectively.%

Assume that $R < (1 - 3r^*)\log c$ \mbox{and pick} $\varepsilon>0$ so that $R + 2\varepsilon < (1 - 3r^*)\log c$. Put $\tilde{A}_T(\bar{u}) := [\min A_T(\bar{u}),\max A_T(\bar{u})]$ and note that $\#\UC_T \leq 2^{(R+\varepsilon)T}$ for all sufficiently large $T$. From \eqref{eq_midpoint} it follows that%
\begin{align*}
 & \limsup_{T \rightarrow \infty} m\bigg( \bigcup_{\bar{u} \in \UC_T} \tilde{A}_T(\bar{u}) \bigg) \leq \limsup_{T \rightarrow \infty}2^{(R+\varepsilon)T}\frac{2b}{c^{(1-3r^*)T}}\\
	&\quad \leq 2b \cdot \limsup_{T \rightarrow \infty} \frac{2^{-\varepsilon T} 2^{(1-3r^*)T\log c}}{c^{(1-3r^*)T}} = 2b \cdot \limsup_{T \rightarrow \infty}2^{-\varepsilon T} = 0.%
\end{align*}
Since $\pi_0 \ll m$, it follows that $\pi_0(\bigcup_{\bar{u} \in \UC_T}\tilde{A}_T(\bar{u})) \rightarrow 0$ as well and thus%
\begin{align*}
 & \limsup_{T \rightarrow \infty} P_{\mathrm{i},\mathrm{c}}\Bigl( |x_0 - \hat{x}_0(T,\bar{u})| \leq \frac{b}{c^{(1-3r^*)T}} \Bigr)\\
&\leq \limsup_{T \rightarrow \infty} P_{\mathrm{i}}\Bigl( x_0 \in \bigcup_{\bar{v} \in \UC_T} \tilde{A}_T(\bar{v}) \Bigr) \\
&= \limsup_{T \rightarrow \infty} \pi_0\Bigl(\bigcup_{\bar{v} \in \UC_T} \tilde{A}_T(\bar{v})\Bigr) = 0.%
\end{align*}
The inequality above holds, since $|x_0 - \hat{x}_0(T,\bar{u})| \leq \frac{b}{c^{(1-3r^*)T}}$ implies the existence of some $\bar{v} \in U^T$ with $x_0 \in \tilde{A}_T(\bar{v})$. Thus, \eqref{eq_671} holds, since%
\begin{equation*}
  \lim_{T \rightarrow \infty}P_{\mathrm{i},\mathrm{c}}\Bigl( |x_0 - \hat{x}_0(T,\bar{u})| > \frac{b}{c^{(1-3r^*)T}} \Bigr) = 1,%
\end{equation*}
independently of the noise realization $\bar{w}_*$.%

{\bf Case 2:} Assume that the control rate satisfies $R \geq (1 - 3r^*)\log c$ and%
\begin{equation*}
  \limsup_{T \rightarrow \infty}P_{\mathrm{i},\mathrm{n},\mathrm{c}}\left( |x_0 - \hat{x}_0(T,\bar{u})| > \frac{b}{c^{(1-3r^*)T}}\right) < \frac{\alpha}{r^*},%
\end{equation*}
contrary to \eqref{eq_671}. Fix a noise realization $\bar{w}_*$ so that%
\begin{equation}\label{eq_eqnCodSuccess}
  \limsup_{T \rightarrow \infty}P_{\mathrm{i},\mathrm{n},\mathrm{c}}\left( |x_0 - \hat{x}_0(T,\bar{u})| > \frac{b}{c^{(1-3r^*)T}}\Bigl| \bar{w} = \bar{w}_* \right) < \frac{\alpha}{r^*},%
\end{equation}
and drop the realization $\bar{w}_*$ in the notation, as in Case 1. Furthermore, write $P_{\mathrm{i},\mathrm{c}}(|x_0 - \hat{x}_0(T,\bar{u})| > \frac{b}{c^{(1-3r^*)T}})$ for the conditional probability above.%

The rest of Case 2 is subdivided into five steps.%

\emph{Step 1 (Construction of sets of bins)}: For every $T\geq1$, we define $\SC_T := \{\hat{x}_0(T,\bar{u}) : \bar{u} \in \UC_T\}$ and enumerate the elements of $\SC_T$ so that%
\begin{equation*}
  \SC_T = \left\{\bar{x}_1(T),\bar{x}_2(T),\ldots,\bar{x}_{n_1(T)}(T)\right\},%
\end{equation*}
where%
\begin{equation}\label{eq_n1_asymp}
  \limsup_{T \rightarrow \infty}\frac{1}{T}\log n_1(T) = R.%
\end{equation}
We define the following collection of bins:%
\begin{equation}\label{eq_binbdef}
  {\bf B}^T_i := \left\{x_0 \in \R: |x_0 - \bar{x}_i(T)| \leq \frac{b}{c^{(1-3r^*)T}}\right\}%
\end{equation}
for $i = 1,\ldots,n_1(T)$, which are not necessarily disjoint. Each ${\bf B}^T_i$ has the same Lebesgue measure, which we denote by $\rho_T := (2b)/(c^{(1-3r^*)T})$. From \eqref{eq_eqnCodSuccess} it follows that%
\begin{equation*}
  \liminf_{T \rightarrow \infty}P_{\mathrm{i},\mathrm{c}}\left( |x_0 - \hat{x}_0(T,\bar{u})| \leq \frac{b}{c^{(1-3r^*)T}}\right) > 1- \frac{\alpha}{r^*},%
\end{equation*}
for which it must be, by the analysis in Case 1, that%
\begin{equation}\label{eq_boundOnM}
  \liminf_{T \rightarrow \infty} \pi_0\Bigl( \bigcup_{i=1}^{n_1(T)} {\bf B}^T_i \Bigr) > 1 - \frac{\alpha}{r^*}.%
\end{equation}
We want to concentrate on the bins that are completely contained in $K = \mathrm{supp}(\pi_0)$. Since we assume that $K$ is an interval, the bins that are only partially contained in $K$ can contribute only very little measure as $T$ becomes large (their union can have at most twice the Lebesgue measure of a single bin), hence we can ignore them. Now assume that the number of bins that are completely outside of $K$ is $n(T)$, and for simplicity assume that these bins are always the last $n(T)$ bins in the enumeration ${\bf B}^T_1,\ldots,{\bf B}^T_{n_1(T)}$. For large $T$, this implies%
\begin{align*}
  &1 - \frac{\alpha}{r^*} \leq \pi_0\Bigl(\bigcup_{i=1}^{n_1(T)}{\bf B}^T_i\Bigr) = \pi_0\Bigl(K \cap \bigcup_{i=1}^{n_1(T)}{\bf B}^T_i\Bigr)\\
	&= \pi_0\Bigl(\bigcup_{i=1}^{n_1(T)-n(T)} {\bf B}^T_i\Bigr) \leq p_{\max} \cdot m\Bigl(\bigcup_{i=1}^{n_1(T)-n(T)} {\bf B}^T_i\Bigr)\\
	&\leq p_{\max} \cdot (n_1(T) - n(T)) \cdot \frac{2b}{c^{(1-3r^*)T}}.%
\end{align*}
Hence, $n_1(T) - n(T)$ must grow at an exponential rate of at least $(1-3r^*)\log c$, just as $n_1(T)$. We will thus, in the rest of the proof, assume w.l.o.g.~that all bins ${\bf B}^T_i$ are completely contained in $K$.%

\begin{figure}[h]
\begin{center}
\includegraphics[height=4.0cm,width=8.0cm]{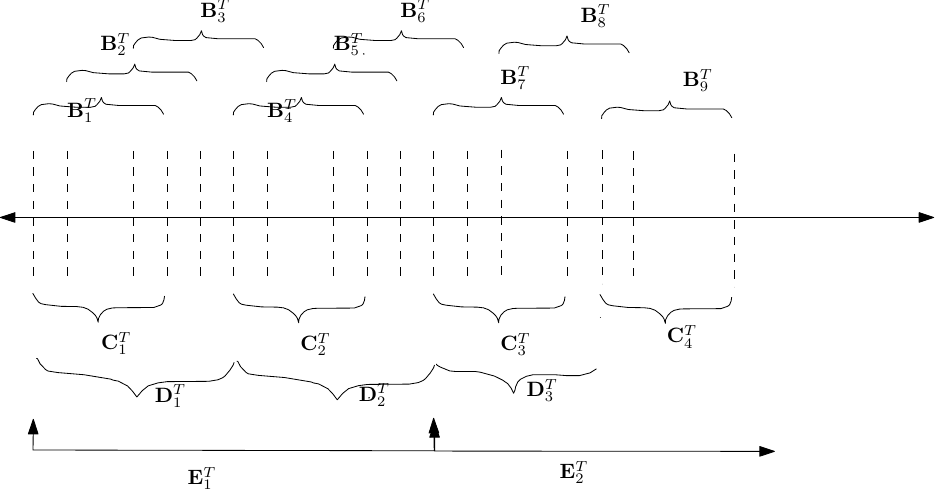}
\caption{Sample construction of bins, with $L=2$\label{Bins1}}
\end{center}
\end{figure}

Now, from $\{{\bf B}^T_i\}$ we extract a subcollection of disjoint bins $\{{\bf C}_i^T\}_{i=1}^{n_2(T)}$ via the construction in Lemma \ref{lem_intervals} (see Figure \ref{Bins1} for an example representation). In particular, we assume that the bins ${\bf B}_i^T$ are ordered according to the natural (non-decreasing) order of their left endpoints. This implies%
\begin{equation}\label{eq_cprops}
  \bigcup_{i=1}^{n_2(T)}{\bf C}^T_i \subset \bigcup_{i=1}^{n_1(T)} {\bf B}^T_i,\ m\Bigl(\bigcup_{i=1}^{n_2(T)}{\bf C}^T_i\Bigr) \geq \frac{1}{2}m\Bigl(\bigcup_{i=1}^{n_1(T)}{\bf B}^T_i\Bigr).%
\end{equation}
Furthermore, it must be that%
\begin{equation}\label{eq_n2_asymp}
  \limsup_{T \rightarrow \infty}\frac{1}{T}\log n_2(T) \geq (1 - 3r^*)\log c,%
\end{equation}
for otherwise, by the analysis in Case 1, $m(\bigcup_{i=1}^{n_2(T)}{\bf C}^T_i) \rightarrow 0$ in contradiction to \eqref{eq_boundOnM} and \eqref{eq_cprops}. Now, using the definition \eqref{eq_leftover} of the leftover set, we define a collection of $n_2(T)$ sets%
\begin{align*}
  {\bf D}^T_k := {\bf C}^T_k \cup L(i_k,i_{k+1}),\quad {\bf D}^T_{n_2(T)} := {\bf C}^T_{n_2(T)}.%
\end{align*}
Hence, ${\bf D}^T_k \subset [\alpha_k,\alpha_{k+1})$, where $\alpha_k = \min {\bf C}^T_k$. The sets ${\bf D}^T_k$ are thus pairwise disjoint. Also observe that%
\begin{equation}\label{eq_dlm}
  m({\bf D}^T_k \backslash {\bf C}^T_k) \leq m({\bf C}^T_k) = \rho_T,%
\end{equation}
since the leftover set has at most the Lebesgue measure of one bin. Finally, for a fixed $L \in \N$, group each collection of $L$ successive ${\bf D}^T_k$ bins as%
\begin{equation*}
  {\bf E}^T_n := \bigcup_{k=(n-1)L+1}^{nL} {\bf D}^T_k,\ n = 1,2,\ldots,\left\lfloor\frac{n_2(T)}{L}\right\rfloor + 1 =: n_3(T).%
\end{equation*}
(In the definition of the last bin ${\bf E}^T_{n_3(T)}$, we add some empty sets to the collection $\{{\bf D}^T_k\}$). From \eqref{eq_n2_asymp} it follows that the number of these bins also satisfies%
\begin{equation}\label{eq_n3_asymp}
  \limsup_{T \rightarrow \infty}\frac{1}{T}\log_2 n_3(T) \geq (1 -3r^*)\log c.%
\end{equation}
Also observe that%
\begin{equation}\label{eq_elm}
  m({\bf E}_n^T) \geq L\rho_T.%
\end{equation}
Let%
\begin{equation*}
  M_T := \bigcup_{i=1}^{n_1(T)} {\bf B}^T_i,\quad
  \overline{M}_T := \bigcup_{i=1}^{n_3(T)} {\bf E}^T_i  \setminus \left({\bf D}^T_{iL} \setminus {\bf C}^T_{iL}\right)%
\end{equation*}
and observe that%
\begin{equation}\label{eq_mtlm}
   m(M_T) \leq 2 n_2(T) \rho_T \leq 2n_3(T) L \rho_T.%
\end{equation}

\emph{Step 2 (The auxiliary coding scheme)}: We now construct an auxiliary coding scheme (in a traditional information-theoretic sense) as follows: We use the received channel output/control sequence to reconstruct the index $\Upsilon$ of the bin ${\bf E}^T_{\Upsilon}$ containing $x_0$ by looking at the points $\hat{x}_0(T,\bar{u})$. With $\hat{\Upsilon}$ denoting the estimate of $\Upsilon$ at the decoder, in the following we study $P(\hat{\Upsilon} \neq \Upsilon)$. By construction of the bins, if%
\begin{equation*}
  x_0 \in \overline{M}_T \wedge |x_0 - \hat{x}_0(T,\bar{u})| \leq \frac{b}{c^{(1-3r^*)T}},%
\end{equation*}
there is no ambiguity, hence $\Upsilon$ can be reconstructed and $\hat{\Upsilon} = \Upsilon$ (no error).%

On the other hand, if $x_0 \in M_T \setminus \overline{M}_T$, we have the following analysis: For every $x_0 \in M_T \setminus \overline{M}_T$, there is $k \geq 1$ so that $x_0 \in {\bf D}^T_{kL} \setminus {\bf C}^T_{kL}$ and hence, given the event $|x_0 - \hat{x}_0(T,\bar{u})| \leq \frac{b}{c^{(1-3r^*)T}}$, $x_0 \in {\bf D}^T_{kL} \setminus {\bf C}^T_{kL}$, the correct bin could be either ${\bf E}^T_k$ or ${\bf E}^T_{k+1}$. So, we can randomly and independently assign the channel output/control to either $\Upsilon=k$ or $\Upsilon=k+1$. The associated error probability is at most $1/2$ when the events $|x_0 - \hat{x}_0(T,\bar{u})| \leq \frac{b}{c^{(1-3r^*)T}}$ and $x_0 \in {\bf D}^T_{kL} \setminus {\bf C}^T_{kL}$ hold, i.e.,%
\begin{equation*}
  P_{\mathrm{i},\mathrm{c}}\Bigl(\hat{\Upsilon} \neq \Upsilon|x_0 \in M_T \backslash \overline{M}_T \wedge |x_0 - \hat{x}_0(T,\bar{u})| \leq \frac{b}{c^{(1-3r^*)T}}\Bigr) \leq \frac{1}{2}.%
\end{equation*}
Altogether, the error probability in our coding scheme can be estimated as follows:%
\begin{align*}
  & P(\hat{\Upsilon} \neq \Upsilon) \leq P_{\mathrm{i},\mathrm{c}}\left(|x_0 - \hat{x}_0(T,\bar{u})| > \frac{b}{c^{(1-3r^*)T}}\right)\\
	& + P_{\mathrm{i},\mathrm{c}}\left(\hat{\Upsilon} \neq \Upsilon||x_0 - \hat{x}_0(T,\bar{u})| \leq \frac{b}{c^{(1-3r^*)T}} \wedge x_0 \in M_T \backslash \overline{M}_T\right)\\
	& \quad \times P_{\mathrm{i},\mathrm{c}}\left(|x_0 - \hat{x}_0(T,\bar{u})| \leq \frac{b}{c^{(1-3r^*)T}} \wedge x_0 \in M_T \backslash \overline{M}_T\right)\\
	& \leq P_{\mathrm{i},\mathrm{c}}\left(|x_0 - \hat{x}_0(T,\bar{u})| > \frac{b}{c^{(1-3r^*)T}}\right) + \frac{1}{2}\pi_0(M_T \backslash \overline{M}_T).%
\end{align*}
From \eqref{eq_eqnCodSuccess} it follows that for all large enough $T$:%
\begin{equation}\label{eq_error_prob_bound}
  P(\hat{\Upsilon} \neq \Upsilon) \leq \frac{\alpha}{r^*} + \frac{1}{2}\pi_0(M_T \backslash \overline{M}_T).%
\end{equation}
Combining \eqref{eq_dlm} and \eqref{eq_elm}, we obtain%
\begin{align}\label{eq_pi0sde}
\begin{split}
  \pi_0({\bf D}^T_{iL} \backslash {\bf C}^T_{iL}) &\leq p_{\max} m({\bf D}^T_{iL} \backslash {\bf C}^T_{iL})\\
	&\leq \frac{p_{\max}}{L}m({\bf E}_i^T) \leq \frac{1}{L} \frac{p_{\max}}{p_{\min}} \pi_0({\bf E}_i^T).%
\end{split}
\end{align}
Since $\overline{M}_T \subset M_T$, the union in the definition of $\overline{M}_T$ is a disjoint union and the union of all ${\bf E}^T_i$ equals $M_T$, we have%
\begin{align*}
  &\pi_0(M_T \backslash \overline{M}_T) = \pi_0(M_T) - \pi_0(\overline{M}_T)\\
	                                     &= \pi_0(M_T) - \sum_i \pi_0\left({\bf E}^T_i  \setminus \left({\bf D}^T_{iL} \setminus {\bf C}^T_{iL}\right)\right)\\
																			 &= \pi_0(M_T) - \sum_i \left(\pi_0({\bf E}^T_i) - \pi_0({\bf D}^T_{iL} \setminus {\bf C}^T_{iL})\right)\\		
												&\stackrel{\eqref{eq_pi0sde}}{\leq} 
									\pi_0(M_T) - \sum_i \Bigl(\pi_0({\bf E}^T_i) - \frac{1}{L} \frac{p_{\max}}{p_{\min}} \pi_0({\bf E}_i^T)\Bigr)\\
									&= \pi_0(M_T) - \Bigl(1 - \frac{1}{L}\frac{p_{\max}}{p_{\min}}\Bigr) \sum_i \pi_0({\bf E}_i^T)
									= \frac{1}{L}\frac{p_{\max}}{p_{\min}} \pi_0(M_T).%
\end{align*}
Together with \eqref{eq_error_prob_bound}, we thus obtain%
\begin{equation}\label{eq_error_prob_est}
  \sum_{i=1}^{n_3(T)} P_{\mathrm{i}}(\Upsilon = i) P_{\mathrm{i},\mathrm{c}}(\hat{\Upsilon} \neq \Upsilon | \Upsilon = i) \leq \frac{\alpha}{r^*} + \frac{1}{2L}\frac{p_{\max}}{p_{\min}} \pi_0(M_T).%
\end{equation}

\emph{Step 3 (Introduction of an auxiliary source variable with uniform distribution)}: From \eqref{eq_mtlm} it follows that%
\begin{equation*}
  \pi_0(M_T) \leq p_{\max} m(M_T) \leq 2n_3(T)p_{\max}L\rho_T.%
\end{equation*}
Since clearly $\pi_0({\bf E}^T_i) \geq p_{\min}\frac{L\rho_T}{\pi_0(M_T)}\pi_0(M_T)$, we obtain%
\begin{equation*}
  P_{\mathrm{i}}(\Upsilon = i) = \pi_0({\bf E}^T_i) \geq \frac{1}{n_3(T)}\frac{p_{\min}}{2p_{\max}}\pi_0(M_T).%
\end{equation*}
Combining this with \eqref{eq_error_prob_est} leads to%
\begin{equation}\label{eq_fest}
  \sum_{i=1}^{n_3(T)}\frac{1}{n_3(T)}P_{\mathrm{i},\mathrm{c}}(\hat{\Upsilon} \neq \Upsilon | \Upsilon = i) \leq \frac{\frac{\alpha}{r^*} + \frac{1}{2L} \frac{p_{\max}}{2p_{\min}}\pi_0(M_T)}{\frac{p_{\min}}{2p_{\max}}\pi_0(M_T)}.%
\end{equation}
Let $W$ be an auxiliary random variable on $\{1,\ldots,n_3(T)\}$ with uniform distribution. Then we have%
\begin{align*}
  &P(\hat{\Upsilon} = W|W = i) = \sum_k P(\hat{\Upsilon} = W \wedge \Upsilon = k|W = i)\\
	&\geq P(\hat{\Upsilon} = W \wedge \Upsilon = i|W = i)	= P(\hat{\Upsilon} = W \wedge \Upsilon = W|W = i)\\
	&= P(\hat{\Upsilon} = W|\Upsilon = W \wedge W = i) P(\Upsilon = W|W = i)\\
	&= P(\hat{\Upsilon} = \Upsilon | \Upsilon = W \wedge W = i) P(\Upsilon = W| W = i)\\
	&= P(\hat{\Upsilon} = \Upsilon | \Upsilon = i) P(\Upsilon = W| W = i).%
\end{align*}
Considering the complementary events, we obtain%
\begin{align*}
  &P(\hat{\Upsilon} \neq W|W = i)\\
	&\leq 1 - (1 - P(\hat{\Upsilon} \neq \Upsilon|\Upsilon = i))(1 - P(\Upsilon \neq W|W = i))\\
	&\leq P(\Upsilon \neq W|W = i) + P(\hat{\Upsilon} \neq \Upsilon|\Upsilon = i).%
\end{align*}
Combining this with \eqref{eq_fest} leads to%
\begin{align}\label{eq_uniform_errorbound}
\begin{split}
  &\sum_{i=1}^{n_3(T)}P(W = i)P(\hat{\Upsilon} \neq W|W = i) \\
	&\leq P(\Upsilon \neq W) + \frac{\frac{\alpha}{r^*} + \frac{1}{2L} \frac{p_{\max}}{2p_{\min}}\pi_0(M_T)}{\frac{p_{\min}}{2p_{\max}}\pi_0(M_T)}.%
\end{split}
\end{align}

\emph{Step 4 (Application of optimal transport theory and coupling of the uniform source with the distribution of $\{{\bf E}^T_i\}$)}: The information-theoretic formulation of information transmission assumes that the messages to be transmitted are uniformly distributed. In the final step of our analysis, we relate the messages represented by the indices of the ${\bf E}^T_i$'s with their induced distribution under $\pi_0$ to a uniformly distributed set of messages: Let $P$ be the distribution of the indices of the ${\bf E}^T_i$'s under $\pi_0$ and $P'$ the uniform distribution of $W$, with the same cardinality as the set of ${\bf E}^T_i$'s. There exists a coupling between $P$ and $P'$ so that the expected error is lower bounded by the total variation distance between $P$ and $P'$; by finding a coupling (cf.~\cite[Eq.~(6.11)]{Villani}), we can achieve that%
\begin{align*}
  \beta &:= P(\Upsilon \neq W)\\
	&= \frac{1}{2}\sum_{i=1}^{n_3(T)} |P(i) - P'(i)| = 1 - \sum_{i=1}^{n_3(T)} \min\left\{P(i),P'(i) \right\}.%
\end{align*}
Let us estimate $\beta$. For sufficiently large $T$, we have%
\begin{align*}
  P(i) &= \pi_0({\bf E}^T_i) \geq p_{\min}m({\bf E}^T_i) = p_{\min}\left[n_3(T)m({\bf E}^T_i)\right]\frac{1}{n_3(T)}\\
	&\geq p_{\min} n_2(T)\rho_T \frac{1}{n_3(T)} \stackrel{\eqref{eq_cprops}}{\geq} p_{\min} \frac{1}{2}m(M_T) \frac{1}{n_3(T)}\\
	&\geq \frac{1}{2}\frac{p_{\min}}{p_{\max}}\pi_0(M_T)\frac{1}{n_3(T)} \stackrel{\eqref{eq_boundOnM}}{\geq} \frac{1}{2}\frac{p_{\min}}{p_{\max}}\left(1 - \frac{\alpha}{r^*}\right) \frac{1}{n_3(T)}.%
\end{align*}
Since $P'(i) = \frac{1}{n_3(T)}$, this implies%
\begin{equation*}
  \beta \leq 1 - \frac{1}{2}\frac{p_{\min}}{p_{\max}}\left(1 - \frac{\alpha}{r^*}\right).%
\end{equation*}

\emph{Step 5 (Application of the strong converse)}: In view of all of the above steps, the proposed coding scheme can be used to encode an auxiliary equi-distributed random variable with an asymptotic average probability of error upper bounded by%
\begin{align*}
& \beta + \limsup_{T \rightarrow \infty}\frac{\frac{\alpha}{r^*} + \frac{1}{2L}\frac{p_{\max}}{2p_{\min} }\pi_0(M_T)}{\frac{p_{\min}}{2 p_{\max}}\pi_0(M_T)}\\
& = \beta + \frac{1}{2L}\frac{p_{\max}^2}{p_{\min}^2} + 2\frac{\alpha}{r^*} \frac{p_{\max}}{p_{\min}} \cdot \limsup_{T \rightarrow \infty} \frac{1}{\pi_0(M_T)}\\
&\stackrel{\eqref{eq_boundOnM}}{\leq} \beta + \frac{1}{2L}\frac{p_{\max}^2}{p_{\min}^2} + 2\frac{\alpha}{r^*} \frac{p_{\max}}{p_{\min}} \cdot \frac{1}{1 - \frac{\alpha}{r^*}}\\
& = \beta + \frac{1}{2L}\frac{p_{\max}^2}{p_{\min}^2} + 2\frac{p_{\max}}{p_{\min}} \cdot \frac{\alpha}{r^* - \alpha}\\
&\leq 1 - \frac{1}{2}\frac{p_{\min}}{p_{\max}}\frac{r^* - \alpha}{r^*} + \frac{1}{2L}\frac{p_{\max}^2}{p_{\min}^2} + 2\frac{p_{\max}}{p_{\min}} \cdot \frac{\alpha}{r^* - \alpha}.%
\end{align*}
This error bound can be made strictly smaller than $1$, when $L$ is chosen sufficiently large and $\alpha$ sufficiently small. Thus, we arrive at a contradiction with the strong converse \eqref{eq_StrongConverse}, because the rate of our coding scheme satisfies%
\begin{equation*}
  \limsup_{T \rightarrow \infty}\frac{1}{T}\log n_3(T) \stackrel {\eqref{eq_n3_asymp}}{\geq} (1 - 3r^*)\log c \stackrel{\eqref{eq_contrad_impl}}{>} C.%
\end{equation*}
The proof is complete.%
\end{IEEEproof}
We note the following variation where the initial measure may have non-compact support with a proof sketch.%

\begin{theorem}\label{thm_noisychannelv2}
Consider system \eqref{eq_scalarsys} satisfying \eqref{eq_expanUnif}. Assume that $\pi_0 \ll m$ with $p$ denoting the density with respect to $m$, and that for every $\epsilon > 0$, there exists a compact interval $K_{\epsilon}$ such that, $\pi_0(K_{\epsilon}) \geq 1 - \epsilon$ and with
\begin{equation*}
  p^K_{\min} := \essinf_{x \in K}p(x) > 0, \quad \quad  p^K_{\max} := \esssup_{x \in K}p(x) < \infty,%
\end{equation*}
the following assumption holds:%
\begin{equation}\label{decayR}
  \lim_{\epsilon \to 0} \frac{\int_{\R \backslash K_{\epsilon}} p(x) \rmd x}{p^{K_{\epsilon}}_{\min} } = 0.%
\end{equation}
Then, if the AMS property is achieved via a causal coding and control strategy over a DMC of capacity $C$, we have%
\begin{equation*}
  C \geq \inf_{x\in\R}\log|f'(x)|.%
\end{equation*}
\end{theorem}

\begin{remark}
A sufficient condition for \eqref{decayR} is that $p$ is differentiable, positive everywhere and monotone decreasing in either direction as $|x|$ increases for sufficiently large values of $|x|$, and $\lim_{|x| \to \infty}p'(x)/p(x) = \infty$. This follows from an application of L'Hospital's theorem to the expression%
\begin{equation*}
  \lim_{x \to \infty} \frac{\int_{|s| > x} p(s) \rmd s}{\min(p(x),p(-x))}.%
\end{equation*}
Probability densities which decay faster than an exponential (such as the Gaussian) satisfy this condition. An exponential density (if one-sided, the denominator will just be $p(x)$) keeps this ratio a constant as $|x|$ increases and densities with a heavier tail than an exponential do not satisfy this condition.
\end{remark}

\begin{IEEEproof}
The proof follows almost identically as that of Theorem \ref{thm_noisychannel}: Case 1 follows identically. For Case 2, in the following, fix a sufficiently small $\epsilon$ and a corresponding $K_{\epsilon}$. If \eqref{eq_671} does not hold, then we can instead of \eqref{eq_eqnCodSuccess}, consider%
\begin{align*}
 & \limsup_{T \rightarrow \infty}P_{\mathrm{i},\mathrm{n},\mathrm{c}}\Bigl(x_0 \in K_{\epsilon}, |x_0 - \hat{x}_0(T,\bar{u})|\\
	&\qquad \qquad \qquad \qquad > \frac{b}{c^{(1-3r^*)T}}\Bigl| \bar{w} = \bar{w}_* \Bigr) < \frac{\alpha}{r^*}.%
\end{align*}
We will construct the auxiliary coding scheme by embedding the bins inside $K_{\epsilon}$. We will thus focus on the sub-probability measure defined by the restriction of $\pi_0$ to $K_{\epsilon}$, defined formally as $\pi^{K_{\epsilon}}_0(B):= \pi_0(B \cap K_{\epsilon})$ for every Borel $B$, and thus we replace \eqref{eq_boundOnM} with%
\begin{equation*}
  \liminf_{T \rightarrow \infty} \pi^{K_{\epsilon}}_0\Bigl( \bigcup_{i=1}^{n_1(T)} {\bf B}^T_i \Bigr) > 1 - \frac{\alpha}{r^*} - \epsilon.%
\end{equation*}
The analysis will go through all the way until in Step 5, where the following term needs to be made less than 1:%
\begin{equation*}
  \beta + \limsup_{T \rightarrow \infty}\frac{\frac{\alpha}{r^*} + \epsilon + \frac{1}{2L}\frac{p^{K_{\epsilon}}_{\max}}{2p^{K_{\epsilon}}_{\min} }\pi^{K_{\epsilon}}_0(M_T)}{\frac{p^{K_{\epsilon}}_{\min}}{2 p^{K_{\epsilon}}_{\max}}\pi^{K_{\epsilon}}_0(M_T)}.%
\end{equation*}
The only additional term, when compared with Step 5 of the proof of Theorem \ref{thm_noisychannel}, is the expression $2 \epsilon p^{K_{\epsilon}}_{\max}/(p^{K_{\epsilon}}_{\min}(1 - \frac{\alpha}{r^*} - \epsilon))$. Since $p^{K_{\epsilon}}_{\max}$ is uniformly bounded under the given assumptions, condition \eqref{decayR} ensures that this term can be made arbitrarily small as $\epsilon$ is made small.%
\end{IEEEproof}

\section{Discussion and concluding remarks}%

In this paper, we considered a stochastic stabilization problem for a general controlled stochastic system over a communication channel. For this problem, we developed a new approach derive fundamental lower bounds on information transmission requirements for control over communication channels. These lower bounds are consistent with the bounds obtained earlier via information-theoretic methods and those obtained for more restrictive models (including linear systems). Moreover, the new proofs are more direct and concise and they allow to obtain finer lower bounds for a large class of systems. The lower bounds obtained for the AMS property are expressed in terms of the determinant of the Jacobian of the nonlinear system model and these recover the existing results for the linear system setup as a special case. For noisy channels, our approach has been to develop a method to relate stabilization entropy and channel capacity through a generalization of the strong converse of information theory.%

Achievability results have been obtained for linear systems in \cite{YukselAMSITArxiv,YukselBasarBook} and for nonlinear systems in \cite{yuksel2016stability}. In particular, \cite[Thm.~4.2]{YukselAMSITArxiv} shows that for a linear system with a diagonalizable matrix $A$, controlled over a DMC, the AMS property can be achieved whenever the channel capacity exceeds the log-sum of the unstable eigenvalues. Hence, in this case the lower bounds following from the results in this paper match with the upper bound. For nonlinear systems of the form%
\begin{equation*}
  x_{t+1} = f(x_t,u_t) + w_t%
\end{equation*}
with $f(\cdot,u):\R^N \rightarrow \R^N$ invertible and $C^1$ for every $u$ and $\{w_t\}$ an i.i.d.~sequence of zero-mean Gaussian variables, it is shown in \cite[Thm.~5.1]{yuksel2016stability} that ergodicity (and thus AMS) can be achieved over a over a discrete noiseless channel under the following assumption: There exist a function $\kappa:\R^N \rightarrow \R^M$ with $\kappa(0)=0$ and a constant $a>0$ such that $|f(x,\kappa(z))|_{\infty} \leq a |x - z|_{\infty}$ for all $x,z\in\R^N$. In this case, the minimal required channel capacity $C_0$ satisfies $C_0 \leq N \log(a) + 1$.%

Finally, we want to mention that local exponential orbit complexity of the open-loop system (as opposed to the global unstable behavior imposed in the system models studied in Section \ref{sec_specmodels1}), in general, does not lead to a positive bound on the channel capacity. For instance, if a system of the form%
\begin{equation*}
  x_{t+1} = f(x_t) + u_t + w_t%
\end{equation*}
admits a compact uniformly hyperbolic set for the associated deterministic system $x_{t+1} = f(x_t)$ and the noise amplitude is sufficiently small, it is well-known that the uncontrolled noisy system $x_{t+1} = f(x_t) + w_t$ admits a random hyperbolic set supporting a stationary measure under mild assumptions, cf.~\cite{Liu} (see also the relevant classical theory of positive Harris recurrence \cite{MeynBook,YukMeynTAC2010}). Hence, for an appropriate initial measure $\pi_0$, the uncontrolled system is already AMS, implying that no information transmission at all is necessary.%

\appendix

\begin{lemma}\label{lem_largedev}
Let $\alpha,\beta,r \in (0,1)$ with $\alpha + \beta = 1$. Then%
\begin{align*}
  &\lim_{T \rightarrow \infty}\frac{1}{T}\log \sum_{t=\lceil (1-r)T \rceil}^T {T \choose t} \alpha^t \beta^{T-t} \\
	&= \left\{\begin{array}{rl} H(r) + r\log\beta + (1-r)\log\alpha & \mbox{if } \beta > r,\\
	0 & \mbox{if } \beta \leq r.
\end{array}\right.%
\end{align*}
As a consequence,%
\begin{equation}\label{eq_subset_rate}
  \lim_{T \rightarrow \infty}\frac{1}{T}\log \sum_{t=\lceil (1-r)T \rceil}^T {T \choose t} = H(r)\ \forall r \in \Bigl(0,\frac{1}{2}\Bigr).%
\end{equation}
\end{lemma}

\begin{IEEEproof}
Let $(X_t)_{t\geq0}$ be an i.i.d.~sequence of $\{0,1\}$-valued Bernoulli random variables with associated probability distribution $\bar{Q}(X_t=0) = \beta$, $\bar{Q}(X_t=1) = \alpha$. Then%
\begin{equation*}
  \sum_{t=\lceil (1-r)T \rceil}^T {T \choose t} \alpha^t \beta^{T-t} = P\Bigl(\frac{1}{T}\sum_{t=0}^{T-1} X_t \geq 1 - r\Bigr).%
\end{equation*}
Sanov's theorem (see \cite[Thm.~11.4.1]{Cover}) yields%
\begin{equation*}
  \lim_{T \rightarrow \infty}\frac{1}{T}\log P\Bigl(\frac{1}{T}\sum
_{t=0}^{T-1} X_t \geq 1 - r\Bigr) = -D(\bar{P}^*||\bar{Q}),%
\end{equation*}
where $\bar{P}^*$ is the information projection of $\bar{Q}$ onto $E := \{P : P(1) \geq 1 - r\}$, i.e., the distribution that minimizes%
\begin{equation*}
  D(P||\bar{Q}) = P(0)\log \frac{P(0)}{\beta} + P(1)\log \frac{P(1)}{\alpha}%
\end{equation*}
under the constraint $P(1) \geq 1 - r$. To determine the solution to this minimization problem, we define the function%
\begin{equation*}
  h(t) := t\log \frac{t}{\beta} + (1-t)\log \frac{1-t}{\alpha},\quad h:[0,1] \rightarrow \R,%
\end{equation*}
whose derivative $h'(t) = \log(\frac{\alpha}{\beta}\frac{t}{1-t})$ vanishes if and only if $t = \beta$. Computing the second derivative $h''(t) = (\ln(2)t(1-t))^{-1}$, we see that $h''(\beta) > 0$, hence $h$ has a minimum at $t = \beta$. Due to the constraint $P(1) \geq 1 - r$, this is only relevant if $\beta \leq r$. In this case, the minimizing distribution is $(\bar{P}^*(0),\bar{P}^*(1)) = (\beta,\alpha)$. Otherwise, the minimum is attained at $t = r$ (by monotonicity) and $(\bar{P}^*(0),\bar{P}^*(1)) = (r,1-r)$. This implies the first assertion of the lemma. The identity \eqref{eq_subset_rate} follows by considering $\alpha = \beta = \frac{1}{2}$.%
\end{IEEEproof}

\begin{lemma}\label{lem_intervals}
Let $\{I_1,\ldots,I_r\}$ be a finite collection of compact intervals, each of equal length $|I_i| = l$. Then there exists a pairwise disjoint subcollection $\{I_{i_1},\ldots,I_{i_k}\}$ satisfying%
\begin{equation*}
  m\Bigl( \bigcup_{j=1}^k I_{i_j} \Bigr) \geq \frac{1}{2}m\Bigl( \bigcup_{i=1}^r I_i \Bigr).%
\end{equation*}
\end{lemma}

\begin{IEEEproof}
We may assume that the intervals $I_i$ are ordered so that their left endpoints form a non-decreasing sequence. Then the indexes $i_j$ are determined as follows: Put $i_1 := 1$. Then take the next interval in $\{I_i\}_{i=2}^r$, which does not intersect $I_{i_1}$ and call it $I_{i_2}$. Let $I_j = [\alpha_j,\beta_j]$. The \emph{leftover space} $L(i_1,i_2)$ between $I_{i_1} = I_1$ and $I_{i_2}$ is%
\begin{equation}\label{eq_leftover}
  L(i_1,i_2) := (\beta_{i_1},\alpha_{i_2}) \cap \bigcup_{j=1}^r I_j%
\end{equation}
and has Lebesgue measure $\leq l$, for otherwise $I_{i_2}$ would not be the first interval not intersecting $I_{i_1}$. Continuing in this way, we find the desired collection of pairwise disjoint intervals and it follows that $m(\bigcup_{j=1}^k I_{i_j}) = kl$, while%
\begin{equation*}
  m\Bigl( \bigcup_{i=1}^r I_i \Bigr) - m\Bigl(\bigcup_{j=1}^k I_{i_j} \Bigr) = m\Bigl( \bigcup_{i=1}^r I_i \backslash \bigcup_{j=1}^k I_{i_j} \Bigr) \leq kl,%
\end{equation*}
implying $2m(\bigcup_{j=1}^k I_{i_j}) = 2kl \geq m\Bigl( \bigcup_{i=1}^r I_i \Bigr)$.
\end{IEEEproof}

\end{document}